\documentclass[12pt,reqno]{amsart}
\usepackage[cp1251]{inputenc}
\usepackage[notref,notcite]{}
\usepackage{amssymb}
\usepackage{amsmath}

\usepackage{amsfonts,xr,graphicx}

\makeatletter

\@addtoreset{equation}{section}
\makeatother

\newtheorem{theorem}{Theorem}[section]
\newtheorem{lemma}[theorem]{Lemma}
\newtheorem{proposition}[theorem]{Proposition}
\newtheorem{corollary}[theorem]{Corollary}

\def\ep{\varepsilon}

\def\D{\mathbb D}
\def\C{\mathbb C}
\def\R{\mathbb R}
\def\S{\mathbb S}
\def\N{\mathbb N}
\def\pa{\partial}
\def\b{\backslash}

\def\H{\mathcal H}
\def\l{\langle}
\def\r{\rangle}
\def\Z{\mathbb Z}
\def\T{{\rm Tr}}
\def\K{\mathcal K}
\def\B{\mathcal B}

\begin{document}

\title[Estimate of the Steklov zeta function on $(-1,1)$]{An estimate for the Steklov zeta function of a planar domain derived from a first variation formula}
\author{Alexandre Jollivet and Vladimir Sharafutdinov}

\thanks{The first author is partially supported by the PRC n\textsuperscript{o} 2795 CNRS/RFBR : Probl\`emes inverses et int\'egrabilit\'e.}

\thanks{The second author is supported by Mathematical Center in Akademgorodok, the
agreement with Ministry of Science and High Education of the Russian
Federation number  075-15-2019-1613.}

\address{Laboratoire de Math\'ematiques Paul Painlev\'e,
CNRS UMR 8524/Universit\'e Lille 1 Sciences et Technologies,
59655 Villeneuve d'Ascq Cedex, France}
\email{alexandre.jollivet@math.univ-lille1.fr}

\address{Sobolev Institute of Mathematics, 4 Koptyug Avenue, Novosibirsk, 630090, Russia}
\address{Novosibirsk State University, 2 Pirogov street, Novosibirsk, 630090, Russia}
\email{sharaf@math.nsc.ru}



\maketitle

\noindent
\begin{abstract}
We consider the Steklov zeta function
$\zeta_\Omega$
of a smooth bounded simply connected planar domain
$\Omega\subset \R^2$
of perimeter
$2\pi$.
We provide a first variation formula for
$\zeta_\Omega$
under a smooth deformation of the domain. On the base of the formula, we prove that, for every
$s\in (-1,0)\cup(0,1)$,
the difference
$\zeta_\Omega(s)-2\zeta_R(s)$
is non-negative and is equal to zero if and only if
$\Omega$
is a round disk
($\zeta_R$
is the classical Riemann zeta function). Our approach gives also an alternative proof of the inequality
$\zeta_\Omega(s)-2\zeta_R(s)\ge0$
for
$s\in (-\infty,-1]\cup(1,\infty)$;
the latter fact was proved in our previous paper [2018] in a different way. We also provide an alternative proof of the equality $\zeta'_\Omega(0)=2\zeta'_R(0)$
obtained by Edward and Wu [1991].
\end{abstract}

\section{Introduction}

Let
$\Omega$
be a simply connected
planar domain bounded by a
$C^\infty$-smooth closed curve
$\partial\Omega$.
The {\it Dirichlet-to-Neumann operator} of the domain
$$
\Lambda_\Omega:C^\infty(\partial\Omega)\rightarrow C^\infty(\partial\Omega)
$$
is defined by
$\Lambda_\Omega f=\left.\frac{\partial u}{\partial\nu}\right|_{\partial\Omega}$,
where
$\nu$
is the outward unit normal to
$\partial\Omega$
and
$u$
is the solution to the Dirichlet problem
$$
\Delta u=0\quad\mbox{\rm in}\quad\Omega,\quad u|_{\partial\Omega}=f.
$$
The Dirichlet-to-Neumann operator is a first order pseudodifferential operator. Moreover, it is a non-negative self-adjoint operator with respect to the $L^2$-product
$$
\l u,v\r=\int\limits_{\partial\Omega}u\bar v\,ds,
$$
where
$ds$
is the Euclidean arc length of the curve
$\partial\Omega$.
In particular, the operator
$\Lambda_\Omega$
has a non-negative discrete eigenvalue spectrum
$$
\mbox{\rm Sp}(\Omega)=\{0=\lambda_0(\Omega)<\lambda_1(\Omega)\leq\lambda_2(\Omega)\leq\dots\},
$$
where each eigenvalue is repeated according to its multiplicity. The spectrum is called the {\it Steklov spectrum} of the domain
$\Omega$.
Steklov eigenvalues depend on the size of
$\Omega$
in the obvious manner:
$\lambda_k(c\Omega)=c^{-1}\lambda_k(\Omega)$
for
$c>0$.
Therefore it suffices to consider domains satisfying the {\it normalization condition}
\begin{equation}
\mbox{Length}(\partial\Omega)=2\pi.
                                     \label{1.1}
\end{equation}

Under condition \eqref{1.1}, Steklov eigenvalues have the following asymptotics \cite[Theorem~1]{E}:
\begin{equation}
\lambda_k(\Omega)=\left\lfloor{k+1\over 2}\right\rfloor +O(k^{-\infty})\quad\mbox{as}\quad k\to\infty,
                                     \label{1.2}
\end{equation}
where
$\lfloor x\rfloor$
stands for the integer part of
$x\in\R$.
Due to the asymptotics, the {\it zeta function of the domain}
$\Omega$
$$
\zeta_\Omega(s)=\mbox{\rm Tr}[\Lambda_\Omega^{-s}]=\sum\limits_{k=1}^\infty\big(\lambda_k(\Omega)\big)^{-s}
$$
is well defined for
$\Re s>1$.
Then
$\zeta_\Omega$
extends to a meromorphic function on
${\mathbb C}$
with the unique simple pole at
$s=1$.
Moreover, the difference
$\zeta_\Omega(s)-2\zeta_R(s)$
is an entire function \cite{E}, where
$\zeta_R(s)=\sum_{n=1}^\infty n^{-s}$
is the classical Riemann zeta function. Observe also that
$\zeta_\Omega(s)$
is real for a real
$s$.

The main result of the present paper is the following

\begin{theorem} \label{Th1.1}
For a smooth simply connected bounded planar domain
$\Omega$
satisfying the normalization condition \eqref{1.1}, the inequality
\begin{equation}
\zeta_\Omega(s)-2\zeta_R(s)\ge0
                                     \label{1.3}
\end{equation}
holds for every real
$s$.
Moreover, if the equality in \eqref{1.3} holds for some real
$s\neq 0$,
then
$\Omega$
is the round disk of radius 1.
\end{theorem}

Inequality \eqref{1.3} was proved for a real
$s$
satisfying
$|s|\ge1$
in \cite[Theorem 1.1]{JS2}. We present a proof of Theorem \ref{Th1.1} which is independent of \cite{JS2} but heavily depends on the compactness arguments of \cite{JS3}.

As a corollary of Theorem \ref{Th1.1} and of the equality
$\zeta_\Omega(0)-2\zeta_R(0)=0$,
we obtain an alternative proof of the equality
$\zeta_\Omega'(0)=2\zeta_R'(0)$
obtained in \cite{EW}.

\bigskip

Now, we discuss an alternative approach to the same problem which is of a more analytical character.

Let
${\mathbb S}=\partial{\mathbb D}=\{e^{i\theta}\}\subset{\mathbb C}$
be the unit circle. The Dirichlet-to-Neumann operator of the unit disk
${\mathbb D}=\{(x,y)\mid x^2+y^2\leq1\}$
will be denoted by
$\Lambda:C^\infty({\mathbb S})\rightarrow C^\infty({\mathbb S})$,
i.e.,
$\Lambda=\Lambda_{\mathbb D}$.
The alternative definition of the operator is given by the formula
$\Lambda e^{in\theta}=|n|e^{in\theta}$
for an integer
$n$.
For a function
$b\in C^\infty({\mathbb S})$,
we write
$b(\theta)$
instead of
$b(e^{i\theta})$
and use the same letter
$b$
for the operator
$b:C^\infty({\mathbb S})\rightarrow C^\infty({\mathbb S})$
of multiplication by the function
$b$.

Given a positive function
$a\in C^\infty({\mathbb S})$,
the operator
$\Lambda_a=a^{1/2}\Lambda a^{1/2}$
has the non-negative discrete eigenvalue spectrum
$$
\mbox{\rm Sp}(\Lambda_a)=\{0=\lambda_0(a)<\lambda_1(a)\leq\lambda_2(a)\leq\dots\}
$$
which is called the {\it Steklov spectrum of the function}
$a$
(or of the operator
$\Lambda_a$).

Two kinds of the Steklov spectrum are related as follows. Given a smooth simply connected planar domain
$\Omega$,
choose a biholomorphism
$\Phi:{\mathbb D}\rightarrow\Omega$
and define the function
$0<a\in C^\infty({\mathbb S})$
by
$a(\theta)=|\Phi'(e^{i\theta})|^{-1}$.
Let
$\phi:{\mathbb S}\rightarrow\partial\Omega$
be the restriction of
$\Phi$
to
${\mathbb S}$.
Then
$\Lambda_a=a^{-1/2}\phi^*\Lambda_\Omega\,\phi^{*-1}a^{1/2}$
and
$\mbox{\rm Sp}(\Lambda_a)=\mbox{\rm Sp}(\Omega)$.
Two latter equalities make sense for an arbitrary positive function
$a\in C^\infty({\mathbb S})$
if we involve multi-sheet domains into our consideration. See \cite[Section 3]{JS} for details. Theorem \ref{Th1.1} is true for multi-sheet domains as well. The normalization condition \eqref{1.1} is written in terms of the function
$a$
as follows:
\begin{equation}
\frac{1}{2\pi}\int\limits_0^{2\pi}\frac{d\theta}{a(\theta)}=1.
                                     \label{1.4}
\end{equation}

The biholomorphism
$\Phi$
of the previous paragraph is defined up to a conformal transformation of the disk
$\D$,
this provides examples of functions with the same Steklov spectrum. Two functions
$a,b\in C^\infty({\mathbb S})$
are said to be {\it conformally equivalent}, if there exists a conformal or anticonformal transformation
$\Psi$
of the disk
${\mathbb D}$
such that
$b=|d\psi/d\theta|^{-1}a\circ\psi$,
where the function
$\psi(\theta)$
is defined by
$e^{i\psi(\theta)}=\Psi(e^{i\theta})$
($\Psi$
is anticonformal if
$\bar\Psi$
is conformal).
If two positive functions
$a,b\in C^\infty({\mathbb S})$
are conformally equivalent, then
$\mbox{Sp}(a)=\mbox{Sp}(b)$.

Under condition \eqref{1.4}, Steklov eigenvalues
$\lambda_k(a)$
have the same asymptotics \eqref{1.2}.
The {\it zeta function of}
$a$
is defined by
\begin{equation}
\zeta_a(s)=\mbox{\rm Tr}[\Lambda_a^{-s}]=\sum\limits_{k=1}^\infty\big(\lambda_k(a)\big)^{-s}
                                            \label{1.5}
\end{equation}
for
$\Re(s)>1$.
It again extends to a meromorphic function on
${\mathbb C}$
with the unique simple pole at
$s=1$
such that
$\zeta_a(s)-2\zeta_R(s)$
is an entire function.

The analytical version of Theorem \ref{Th1.1} sounds as follows:

\begin{theorem} \label{Th1.2}
For a positive function
$a\in C^\infty({\mathbb S})$
satisfying the normalization condition \eqref{1.4}, the inequality
\begin{equation}
\zeta_a(s)-2\zeta_R(s)\ge0
                                     \label{1.6}
\end{equation}
holds for every real
$s$.
Moreover, the equality in \eqref{1.6} holds for some real
$s\neq 0$
if and only if
$a$
is conformally equivalent to
$\mathbf 1$
(= the constant function identically equal to 1).
\end{theorem}

The second statement of the theorem is not true for
$s=0$
since
$\zeta_a(0)=2\zeta_R(0)=-1$
for every positive function
$a\in C^\infty({\mathbb S})$
satisfying the normalization condition \eqref{1.4}. Observe also that
$\zeta_{\mathbf 1}=2\zeta_{R}$.

Theorems \ref{Th1.1} and \ref{Th1.2} are equivalent if multi-sheet domains are involved into Theorem~\ref{Th1.1} (see \cite{JS2} for instance).

\bigskip

We use the derivative
$D=-i\frac{d}{d\theta}:C^\infty({\mathbb S})\rightarrow C^\infty({\mathbb S})$.
The Hilbert space
$L^2(\S)$
is considered with the standard scalar product
$$
\l u,v\r=\int_{\S}u(\theta)\overline{v(\theta)}\,d\theta.
$$
The {\it Hilbert transform}
$\H$
is the linear operator on
$L^2(\S)$
defined by
$$
\H({\mathbf 1})=0,\quad \H e^{in\theta}={\rm sgn}(n)e^{in\theta}\ \textrm{for an integer}\ n\neq0.
$$
(We emphasize that
$\H$
differs from the operator
$H$
that is also called the Hilbert transform in \cite{JS3}. In particular,
$H$
is a unitary operator while
$\H$
has the one-dimensional kernel consisting of constant functions.)

Our proof of Theorem \ref{Th1.2} is based on a clever deformation of the function
$a$.
A real function
$\alpha\in C^\infty\big((-\varepsilon,\varepsilon)\times{\mathbb S}\big)$
is called a {\it deformation} (or {\it variation}) of a positive function
$a\in C^\infty({\mathbb S})$
if
$\alpha(0,\theta)=a(\theta)$.
For such a deformation; the function
$\alpha_\tau\in C^\infty({\mathbb S})$,
defined by
$\alpha_\tau(\theta)=\alpha(\tau,\theta)$,
is positive for sufficiently small
$|\tau|$.
Without lost of generality (choosing a smaller
$\varepsilon>0$)
we will assume that
$\alpha_\tau$
is positive for all
$\tau\in(-\varepsilon,\varepsilon)$.
Then the zeta function
$\zeta_{\alpha_\tau}$
is well defined. In Sections 2--3 we will prove that
$\zeta_{\alpha_\tau}(z)$
smoothly depends on
$(z,\tau)$
for
$1\neq z\in{\mathbb C}$
and will compute the derivative
$\frac{\partial\zeta_{\alpha_\tau}(z)}{\partial\tau}$
(Lemma \ref{L3.6}). We will also prove that
$\zeta_{\alpha_\tau}(z)$
is continuous in
$\tau$
for
$1\neq z\in{\mathbb C}$
when
$\alpha_\tau$
(belonging to
$C^\infty(\S)$)
is only continuous in
$\tau$.
The rest of the paper is devoted to the proof of the following statement

\begin{theorem} \label{Th1.3}
Given a positive function
$a\in C^\infty({\mathbb S})$
satisfying the normalization condition \eqref{1.4}, there exists a deformation
$\alpha_\tau\ (0\le\tau<\infty)$
of
$a$
such that

{\rm (1)} for every
$\tau\in[0,\infty)$,
the function
$\alpha_\tau$
is positive and satisfies the same normalization condition
\begin{equation}
\int\limits_{\mathbb S}\alpha_\tau^{-1}(\theta)\,d\theta=2\pi;
                                     \label{1.7}
\end{equation}

{\rm (2)} the deformation satisfies the equation
\begin{equation}
{\pa\alpha_\tau\over \pa\tau}=-\alpha_\tau (\Lambda \alpha_\tau)+(\H\alpha_\tau) (D\alpha_\tau)\quad\textrm{for}\quad \tau\in[0,\infty);
                                       \label{1.8}
\end{equation}

{\rm (3)} the derivative
$\frac{\partial\zeta_{\alpha_\tau}(s)}{\partial\tau}$
is non-positive for every real
$s$
and every
$\tau\in[0,\infty)$;

{\rm (4)}
$\alpha_\tau$
converges to
$\mathbf 1$
as
$\tau\rightarrow\infty$
in the
$C^\infty$-topology of
$C^\infty({\mathbb S})$.

Moreover, if
$\ \frac{\partial\zeta_{\alpha_\tau}(s)}{\partial\tau}=0$
for some
$0\neq s\in{\R}$
and for all
$\tau\in[0,\infty)$,
then
$a$
is conformally equivalent to
$\mathbf 1$.
\end{theorem}

{\bf Remark.} The right-hand side of the formula
$$
\frac{\partial\zeta_{\alpha_\tau}(s)}{\partial\tau}=\frac{\partial}{\partial\tau}\Big(\zeta_{\alpha_\tau}(s)-2\zeta_R(s)\Big)
$$
makes sense for any
$s\in\C$
since
$\zeta_{\alpha_\tau}-2\zeta_R$
is an entire function. In virtue of the formula, the derivative
$\frac{\partial\zeta_{\alpha_\tau}(s)}{\partial\tau}$
is well defined for all
$s\in\C$
although the zeta function
$\zeta_{\alpha_\tau}(s)$
is not defined at the pole
$s=1$.

Theorem \ref{Th1.2} follows from Theorem \ref{Th1.3} as is shown in Section 4.

\section{Asymptotic behavior of eigenvalues and eigenspaces}

\subsection{Uniform asymptotics of the Steklov eigenvalues}

For a positive function
$a\in C^\infty({\S})$,
we introduce the operator
$D_a=a^{1/2}Da^{1/2}$.
Recall also that
$\Lambda_a=a^{1/2}\Lambda a^{1/2}$.

Let
$\alpha_\tau\ (-\varepsilon<\tau<\varepsilon)$
be a deformation of a positive function
$a\in C^\infty({\S})$.
Recall that the function
$\alpha_\tau$
is assumed to be positive for every
$\tau\in(-\varepsilon,\varepsilon)$.
Smooth deformations
$\alpha\in C^\infty\big((-\varepsilon,\varepsilon)\times{\mathbb S}\big)$
are used in the most part of the paper. But in Section 4 for our main results, we will need a continuous deformation
$\alpha\in C^0\big((-\ep,\ep), C^\infty(\S)\big)$
which can be not smooth. Therefore we assume now that
$\alpha\in C^l\big((-\ep,\ep), C^\infty(\S)\big)$
with some integer  $0\le l\le\infty$.
We also assume that both
$a$
and
$\alpha_\tau$
satisfy the normalization conditions \eqref{1.4} and \eqref{1.7} respectively.

Given a deformation
$\alpha_\tau$
of a function
$a$,
we introduce the operators
$$
A_\tau=\Lambda_{\alpha_\tau}^2+I,\quad B_\tau=D_{\alpha_\tau}^2+I,\quad \Delta_\tau=A_\tau-B_\tau,
$$
where
$I$
is the identity operator.
By \cite[Section 5.4]{JS}, the commutator
$[\alpha_\tau,\H]$
is a smoothing operator with the Schwartz kernel
$$
\tilde K (\tau,\theta,\theta')={1\over {2\pi}}\big(\alpha_\tau(\theta)-\alpha_\tau(\theta')\big)\cot{\theta-\theta'\over 2}.
$$
Therefore
$\Delta_\tau=\alpha_\tau^{1/2}D\H[\alpha_\tau,\H]D\alpha_\tau^{1/2}$
is also a smoothing operator with the Schwartz kernel
\begin{equation}
K(\tau,\theta,\theta')=\alpha_\tau^{1/2}(\theta)\Big(\H_\theta D_\theta D_{\theta'} \tilde K (\tau,\theta,\theta')\Big)\alpha_\tau^{1/2}(\theta').
                                                  \label{2.1}
\end{equation}

For
$n\in\Z$
and
$\tau\in(-\ep,\ep)$,
we define the function
$\phi_{n,\tau}\in C^\infty(\S)$
by
\begin{equation}
\phi_{n,\tau}(\theta)=(2\pi)^{-{1/2}}\big(\alpha_\tau(\theta)\big)^{-1/2}\exp\Big[in\int_0^\theta \alpha_\tau^{-1}(s)\,ds\Big].
                                                 \label{2.2}
\end{equation}
By \cite[Lemma 2.1]{JS2},
$\big\{\phi_{n,\tau}\big\}_{n\in\Z}$
is the orthonormal basis of
$L^2(\S)$
consisting of eigenfunctions of the operator
$D_{\alpha_\tau}$
such that
$D_{\alpha_\tau}\phi_{n,\tau}=n\phi_{n,\tau}$.
For
$l\in \N$,
let us denote
$u_{2l,\tau}=\phi_{l,\tau}$
and
$u_{2l+1,\tau}=\phi_{-l,\tau}$ and we also denote $u_0=\phi_{0,\tau}$.

Let
$\mbox{Sp}(\Lambda_{\alpha_\tau})=\{0=\lambda_{0,\tau}<\lambda_{1,\tau}\le\lambda_{2,\tau}\le\dots\}$
be the eigenvalue spectrum of the operator
$\Lambda_{\alpha_\tau}$.
We repeat Edward's argument \cite[Theorem 1]{E} to prove the following statement.

\begin{lemma} \label{L2.1}
Let
$\alpha\in C^0\big((-\ep,\ep), C^\infty(\S)\big)$
be a continuous deformation of a positive function
$a\in C^\infty({\S})$
such that every
$\alpha_\tau$
is also a positive function. Assume
$a$
and
$\alpha_\tau$
to satisfy the normalization conditions \eqref{1.4} and \eqref{1.7} respectively.
Then the following uniform asymptotics holds for every $\ep'\in(0,\ep)$:
$$
\sup_{\tau\in [-\ep',\ep']}\big|\lambda_{k,\tau}-\lfloor {(k+1)/2}\rfloor \big|=O(k^{-\infty})\quad\mbox{\rm as}\quad k\to\infty.
$$
\end{lemma}

\begin{proof}
We recall the following min-max principle. For
$k\in \N$
$$
\lambda_{k,\tau}^2+1=\max_{{\rm codim}\, E_k=k}\ \min_{\phi\in E_k,\, \|\phi\|_{L^2(\S)}=1}\l A_\tau\phi,\phi\r,
$$
$$
\lfloor (k+1)/2\rfloor^2+1=\max_{{\rm codim}\, E_k=k}\ \min_{\phi\in E_k,\, \|\phi\|_{L^2(\S)}=1}\l B_\tau\phi,\phi\r.
$$
For
$\phi\in L^2(\S)$
$$
\l A_\tau\phi,\phi\r=\l B_\tau\phi,\phi\r+\l \Delta_\tau\phi,\phi\r.
$$
Therefore for any subspace
$E_k$
of codimension
$k$
$$
\lambda_{k,\tau}^2+1\ge  \min_{\phi\in E_k,\ \|\phi\|_{L^2(\S)}=1}\l B_\tau\phi,\phi\r+\min_{\phi\in E_k,\ \|\phi\|_{L^2(\S)}=1}\l \Delta_\tau\phi,\phi\r.
$$
Taking
$E_k$
the subspace of codimension
$k$
spanned by the eigenvectors
$\{u_l\}_{l\ge k}$
of
$B_\tau$,
we obtain
$$
\lambda_{k,\tau}^2- \lfloor {(k+1)/2}\rfloor^2 \ge \min_{\phi\in E_k,\ \|\phi\|_{L^2(\S)}=1}\l \Delta_\tau\phi,\phi\r.
$$
Since
$ \Delta_\tau$
is a smoothing operator,
$\Delta_\tau  B_\tau^l$
is a bounded operator for any
$l\in \N$
and its operator norm
is bounded  uniformly in
$\tau\in [-\ep',\ep']$
by some constant
$C_l$
since its Schwartz kernel is a continuous function on
$(-\ep\times \ep)\times \S\times \S$.
Hence for any
$\phi\in E_k$
satisfying
$\|\phi\|_{L^2(\S)}=1$,
$$
\big|\l \Delta_\tau \phi,\phi\r\big|=\big|\l \Delta_\tau  B_\tau^l B_\tau^{-l}\phi,\phi\r\big|
\le C_l\|B_\tau^{-l}\phi\|_{L^2(\S)}.
$$
Then we use that, for a unit vector
$\phi$
of the subspace
$E_k$
spanned by the eigenvectors
$\{u_l\}_{l\ge k}$
of
$B_\tau$,
we have
$\|B_\tau^{-l}\phi\|_{L^2(\S)}\le 4k^{-2l}$.

We have thus obtained the estimate
$$
-\lambda_{k,\tau}^2 + \lfloor {(k+1)/2}\rfloor^2 \le  4C_l\,\,k^{-2l}.
$$

We can transpose the roles played by
$A_\tau$
and
$B_\tau$
to obtain the estimate
$$
\lambda_{k,\tau}^2 - \lfloor {(k+1)/2}\rfloor^2 \le  4 C'_l\,\lambda_{k,\tau}^{-2l},
$$
where
$C'_l$
stands for a uniform bound of the operator norm of
$\Delta_\tau  A_\tau^l$.
\end{proof}

\subsection{Continuity of Steklov eigenvalues in $\tau$.}
Eigenvalues of the compact self-adjoint operator
$$
F_\tau=(\Lambda_{\alpha_\tau}+I)^{-1}:L^2({\S})\rightarrow L^2({\S})
$$
are
$(\lambda_{k,\tau}+1)^{-1}\le 1\ (k\in \N)$.
If one proved that the eigenvalues are continuous in
$\tau$,
then the Steklov eigenvalues
$\lambda_{k,\tau}$
would be also continuous in
$\tau$.
The proof relies on auxiliary lemmas presented below. The lemmas deal with
$C^l$-regularity with respect to
$\tau$
while only the case
$l=0$
is needed for continuity of eigenvalues. However the lemmas will be used  in a broader context in next sections.

Given a
$C^l$-deformation
$\alpha_\tau$
of a positive function
$a$,
the operator
$\Lambda_{\alpha_\tau}:C^\infty(\S)\rightarrow C^\infty(\S)$
depends
$C^l$-smoothly on
$\tau$.
In the case of
$l\ge 1$,
we  differentiate the equality
$\Lambda_{\alpha_\tau}=\alpha_\tau^{1/2}\Lambda\alpha_\tau^{1/2}$
with respect to
$\tau$
to obtain
\begin{equation}
{\pa \Lambda_{\alpha_\tau}\over \pa \tau}={1\over 2}{\pa\ln(\alpha_\tau)\over \pa\tau}\Lambda_{ \alpha_\tau}+{1\over 2}\Lambda_{ \alpha_\tau}{\pa\ln(\alpha_\tau)\over \pa\tau}.
                                             \label{2.3}
\end{equation}
Similar formulas hold for higher order derivatives
${\pa^k \Lambda_{\alpha_\tau}\over \pa \tau^k}$
for
$k\le l$.

Recall that the Sobolev space
$H^s(\S)$
can be defined for
$s\in \R$
as the completion of
$C^\infty(\S)$
with respect to the norm
$
\|f\|_{H^s(\S)}=\|(D^2+1)^{s/2} f\|_{L^2(\S)}.
$
Let
$\mathcal{L}\big(H^{s}(\S),H^{s'}(\S)\big)$
be the Banach space of all bounded operators
$H^{s}(\S)\rightarrow H^{s'}(\S)$
furnished with the operator norm.

\begin{lemma} \label{L2.2}
Let
$l$
be either a non-negative integer or
$\infty$.
Let
$\alpha\in C^l\big((-\ep,\ep), C^\infty(\S)\big)$
be a deformation of a positive function
$a\in C^\infty({\S})$
such that every
$\alpha_\tau$
is also a positive function. Assume
$a$
and
$\alpha_\tau$
to satisfy the normalization conditions \eqref{1.4} and \eqref{1.7} respectively.
Then

{\rm (1)}
For every real
$s$
and for every
$\tau\in(-\ep,\ep)$,
the operator
$\Lambda_{\alpha_\tau}$
belongs to the space
$\mathcal{L}\big(H^{s}(\S),H^{s-1}(\S)\big)$
and the function
$\tau\mapsto\Lambda_{\alpha_\tau}$
belongs to
$C^l\big((-\ep,\ep),\mathcal{L}(H^{s}(\S),H^{s-1}(\S))\big)$.

{\rm (2)}
Similarly, the operator-valued function
$F_\tau\in \mathcal{L}(H^s(\S),H^{s+1}(\S))$
is of the class
$C^l$
in
$\tau$.
\end{lemma}

\begin{proof}
The operator
$ \Lambda_{\alpha_\tau}$
is a partial case of more general operators of the form
$$
f_1(\tau)\Lambda f_2(\tau),
$$
where
$f_i\in C^l\big((-\ep,\ep), C^\infty (\S)\big)\ (i=1,2)$.
The operators of multiplication  by
$f_i(\tau)$
can be extended to bounded operators in any
$H^s(\S)$
and these bounded operators are of class
$C^l$
in
$\tau$.
Note also that
$\Lambda$
is a bounded operator from
$H^s(\S)$
to
$H^{s-1}(\S)$.
Hence the family of bounded operators
$\Lambda_{\alpha_\tau}\in \mathcal{L}(H^s(\S),H^{s-1}(\S))$
is of class
$C^l$
in
$\tau$.
In the case
$l\ge 1$,
the first derivative with respect to
$\tau$
is expressed by the formula
$$
{\pa f_1(\tau)\over \pa \tau}\Lambda f_2(\tau)+f_1(\tau)\Lambda {\pa f_2(\tau)\over \pa \tau}.
$$

Now, we prove the second statement. The operator
$F_\tau$
is the inverse of the operator
$\Lambda_{\alpha_\tau}+I$
which belongs to
$\mathcal{L}(H^{s+1}(\S),H^s(\S))$
and is of the class
$C^l$
in
$\tau$.
Let us explain why
$F_\tau$
is a continuous operator-valued function (the case $l=0$). This is based on the inversion formula by Neumann series:
If
$\tau$
is close enough to
$\tau_0\in(-\ep,\ep)$,
then the norm of the operator
$(\Lambda_{\alpha_\tau}-\Lambda_{\alpha_{\tau_0}})F_{\tau_0}:H^s(\S)\rightarrow H^s(\S)$
is less than 1 and
$$
\Lambda_{\alpha_\tau}+I=\big((\Lambda_{\alpha_\tau}-\Lambda_{\alpha_{\tau_0}})F_{\tau_0}+I\big)(\Lambda_{\alpha_{\tau_0}}+I).
$$
This gives
$$
F_\tau=F_{\tau_0}\sum_{k=0}^{\infty}(-1)^k\big((\Lambda_{\alpha_\tau}-\Lambda_{\alpha_{\tau_0}})F_{\tau_0}\big)^k.
$$
This formula also provides that
$F_\tau$
is of the class
$C^l$
in
$\tau$
when
$l\ge 1$,
and its derivative is given by the formula
$$
{\pa F_\tau\over \pa \tau}=-F_{\tau}{\pa \Lambda_{\alpha_\tau}\over \pa\tau} F_{\tau}.
$$
\end{proof}

In the case of
$l=s=0$,
we apply the min-max principle to the compact self-adjoint operator
$F_\tau$
in
$L^2(\S)$
to obtain the following\footnote{We cannot get a better statement than continuity for the eigenvalues. Take for instance the family of bounded operators
$G_\tau=\left(\begin{array}{cc}
2+\cos(\tau)&0 \\
0&2+\sin(\tau)\\
\end{array}\right)$.
The eigenvalues
$\lambda_+(\tau)=\max(2+\cos(\tau),2+\sin(\tau))$
and
$\lambda_-(\tau)=\min(2+\cos(\tau),2+\sin(\tau))$
are not derivable at
$\tau={\pi/4}$
although the family $G_\tau$ is a smooth family of bounded operators in
$\tau$.
See also \cite[Chapter 2, Section 5, example 5.9, p. 115]{K}.}

\begin{corollary} \label{C2.1}
Under hypotheses of Lemma \ref{L2.2} with
$l=0$,
Steklov eigenvalues
$\lambda_{k,\tau}$
are continuous in
$\tau$.
\end{corollary}

\subsection{Asymptotics of the Steklov eigenvectors}
We still consider a deformation
$\alpha_\tau$
of a positive function
$a\in C^\infty(\S)$
satisfying hypotheses of Lemma \ref{L2.2} with some
$l\in\N\cup\{\infty\}$.
For every
$\tau\in (-\ep,\ep)$,
let
$\{\Psi_{k,\tau}\}_{k\in \N}$
be an orthonormal basis of
$L^2({\S})$
consisting of Steklov eigenvectors for
$\Lambda_{\alpha_\tau}$
such that
$\Lambda_{\alpha_\tau}\Psi_{k,\tau}=\lambda_{k,\tau}\Psi_{k,\tau}$.

For  a positive integer
$k$,
we denote by
$P_{\lfloor {k+1\over 2}\rfloor,\tau}$
the orthogonal projection of
$L^2(\S)$
onto the two-dimensional eigenspace of
$|D_{\alpha_{\tau}}|$
spanned by the vectors
$\phi_{\pm \lfloor {k+1\over 2}\rfloor,\tau}$
that are defined by \eqref{2.2}.
For
$k=0$,
the operator
$P_{0,\tau}$
is the orthogonal projection onto the one-dimensional space spanned by
$\phi_{0,\tau}$.

\begin{lemma} \label{L2.3}
Under hypotheses of Lemma \ref{L2.2} with some
$l\in\N\cup\{\infty\}$,
the following uniform asymptotics holds for every
$\ep'\in(0,\ep)$
and for every
$s\in\R$:
\begin{equation}
\sup_{\tau\in [-\ep',\ep']}\|\Psi_{k,\tau}-P_{\lfloor {k+1\over 2}\rfloor,\tau}\Psi_{k,\tau}\|_{H^s(\S)}=O(k^{-\infty})\quad\mbox{as}\quad k\rightarrow\infty.
                                      \label{2.4}
\end{equation}
\end{lemma}

\begin{proof}
It suffices to prove the statement for
$s=m\in{\N}$.
We start with proving \eqref{2.4} for
$s=0$.

The equality
$\Psi_{k,\tau}=\sum_{p\in \N}P_{p,\tau}\Psi_{k,\tau}$
can be written in the form
$$
\Psi_{k,\tau}-P_{\lfloor {k+1\over 2}\rfloor,\tau}\Psi_{k,\tau}=\sum_{p\not=\lfloor {k+1\over 2}\rfloor}P_{p,\tau}\Psi_{k,\tau}.
$$
Since summands on the right-hand side are orthogonal to each other,
\begin{equation}
\|\Psi_{k,\tau}-P_{\lfloor {k+1\over 2}\rfloor,\tau}\Psi_{k,\tau}\|_{L^2(\S)}^2
=\sum_{p\not=\lfloor {k+1\over 2}\rfloor}\|P_{p,\tau}\Psi_{k,\tau}\|_{L^2(\S)}^2.
                                      \label{2.5}
\end{equation}

For any
$r\in \N$,
$\Delta_{\tau,r}=\Lambda_{\alpha_\tau}^{2r}- D_{\alpha_\tau}^{2r}$
is a smoothing operator whose operator norm is bounded uniformly in
$\tau$.
We rewrite the identity
$\Lambda_{\alpha_\tau}^{2r}\Psi_{k,\tau}=\lambda_{k,\tau}^{2r}\Psi_{k,\tau}$
as follows:
\begin{equation}
\Delta_{\tau,r}\Psi_{k,\tau}=\sum_{p\in \N}(\lambda_{k,\tau}^{2r}-p^{2r})P_{p,\tau}\Psi_{k,\tau}.
                                      \label{2.6}
\end{equation}
Since summands on the right-hand side are orthogonal to each other, \eqref{2.6} implies
$$
\|\Delta_{\tau,r}\Psi_{k,\tau}\|^2
=(\lambda_{k,\tau}^{2r}-\lfloor {(k+1)/2}\rfloor^{2r})^2\|P_{\lfloor{k+1\over 2}\rfloor,\tau}\Psi_{k,\tau}\|^2+\sum_{p\not=\lfloor {k+1\over 2}\rfloor}(\lambda_{k,\tau}^{2r}-p^{2r})^2\|P_{p,\tau}\Psi_{k,\tau}\|^2
$$
(all norms are
$L^2(\S)$-norms).
By Lemma \ref{L2.1}, the first term on the right hand side is bounded uniformly in
$\tau\in[-\ep',\varepsilon']$
for sufficiently large
$k$.
Hence there exist a rank
$N_1$
and constant
$C$
(independent of $k$) such that
\begin{equation}
\sup_{\tau\in[-\ep',\ep']} \sum_{p\not=\lfloor {k+1\over 2}\rfloor}(\lambda_{k,\tau}^{2r}-p^{2r})^2\|P_{p,\tau}\Psi_{k,\tau}\|^2\le C\quad
\mbox{for}\quad k\ge N_1.
                           \label{2.7}
\end{equation}

Let us represent the difference
$\lambda_{k,\tau}^{2r}-p^{2r}$
in the form
$$
\lambda_{k,\tau}^{2r}-p^{2r}=(\lambda_{k,\tau}-p)(\lambda_{k,\tau}^{2r-1}+\lambda_{k,\tau}^{2r-2}p+\dots+p^{2r-1}).
$$
From this
$$
(\lambda_{k,\tau}^{2r}-p^{2r})^2\ge(\lambda_{k,\tau}-p)^2\lambda_{k,\tau}^{4r-2}.
$$
By Lemma \ref{L2.1}, the second factor on the right-hand side is close to
$(k/2)^{4r-2}$
for sufficiently large
$k$
while the first factor is not less than
$1+O(k^{-1})$
for
$p\not=\lfloor (k+1)/2\rfloor$.
In other words, there exist a rank
$N_2$
and constant
$c>0$
(independent of $k$) such that
\begin{equation}
\inf_{\tau \in [-\ep',\ep'],\ p\not=\lfloor {k+1\over 2}\rfloor}(\lambda_{k,\tau}^{2r}-p^{2r})^2\ge c k^{4r-2}\quad
\mbox{for}\quad k\ge N_2.
                              \label{2.8}
\end{equation}

Combining \eqref{2.7} and \eqref{2.8}, we see that
$$
k^{4r-2}\sup_{\tau\in[-\ep',\ep']} \sum_{p\not=\lfloor {k+1\over 2}\rfloor}\|P_{p,\tau}\Psi_{k,\tau}\|^2\le c^{-1}C.
$$
Together with \eqref{2.5}, this implies
$$
k^{4r-2}\sup_{\tau\in[-\ep',\ep']} \|\Psi_{k,\tau}-P_{\lfloor {k+1\over 2}\rfloor,\tau}\Psi_{k,\tau}\|_{L^2({\S})}^2\le c^{-1}C_1.
$$
Since
$r$
is arbitrary, the statement is proved for
$s=0$.

Now, we prove \eqref{2.4} for
$s=m\in{\N}$.
Applying the operator
$D_{\alpha_\tau}^{2m}$
to both sides of \eqref{2.6}, we obtain
$$
D_{\alpha_\tau}^{2m}\Delta_{\tau,r}\Psi_{k,\tau}=\sum_{p\in \N}p^{2m}(\lambda_{k,\tau}^{2r}-p^{2r})P_{p,\tau}\Psi_{k,\tau}.
$$
Since summands on the right-hand side are orthogonal to each other, this implies
$$
\begin{aligned}
\|D_{\alpha_\tau}^{2m}\Delta_{\tau,r}\Psi_{k,\tau}\|^2
&=\lfloor (k+1)/2\rfloor^{4m}(\lambda_{k,\tau}^{2r}-\lfloor (k+1)/2\rfloor^{2r})^2\|P_{\lfloor{k+1\over 2}\rfloor,\tau}\Psi_{k,\tau}\|^2\\
&+\sum_{p\not=\lfloor {k+1\over 2}\rfloor}(\lambda_{k,\tau}^{2r}-p^{2r})^2\|p^{2m}P_{p,\tau}\Psi_{k,\tau}\|^2
\end{aligned}
$$
(all norms are
$L^2(\S)$-norms).
Again
$D_{\alpha_\tau}^{2m}\Delta_{\tau,r}$
is a smoothing operator whose operator norm is bounded uniformly in
$\tau$,
and
$\lfloor (k+1)/2\rfloor^{4m}(\lambda_{k,\tau}^{2r}-\lfloor (k+1)/2\rfloor^{2r})^2$
is also uniformly bounded in
$\tau$
for sufficiently large
$k$.
Applying the same reasoning as above we obtain that there exists a rank
$N$
and constant
$C$
such that
$$
\sup_{\tau\in[-\ep',\ep']}\sum_{p\not=\lfloor {k+1\over 2}\rfloor}\|p^{2m}P_{p,\tau}\Psi_{k,\tau}\|^2\le Ck^{2-4r}\quad
\mbox{for}\quad k\ge N,
$$
i.e.,
\begin{equation}
\sup_{\tau\in[-\ep',\ep']} \|D_{\alpha_\tau}^{2m}\big(\Psi_{k,\tau}-P_{\lfloor {k+1\over 2}\rfloor,\tau}\Psi_{k,\tau}\big)\|_{L^2({\S})}^2\le Ck^{2-4r}\quad
\mbox{for}\quad k\ge N.
                                      \label{2.9}
\end{equation}

We conclude the proof as follows. Given
$m\in \N$,
there exists a constant
$C_1$
such that
\begin{equation}
\|D^{2m}\phi\|_{L^2({\S})}^2\le C_1\big( \|D_{\alpha_\tau}^{2m}\phi\|_{L^2({\S})}^2+\|\phi\|_{L^2({\S})}^2\big)
                         \label{2.10}
\end{equation}
for any $\phi\in C^\infty(\S)$ and any $\tau\in [-\ep',\ep']$.
Combining estimates \eqref{2.9}, \eqref{2.10} and the statement for ``$m=0$", we obtain the existence of
$N_2\in \N$
and constant
$C_2$
such that
$$
\sup_{\tau\in[-\ep',\ep']} \|D^{2m}\big(\Psi_{k,\tau}-P_{\lfloor {k+1\over 2}\rfloor,\tau}\Psi_{k,\tau}\big)\|^2\le C_2k^{2-4r}\quad
\mbox{for}\quad k\ge N_2.
$$
Since
$r$
is arbitrary, the asymptotics \eqref{2.4} is proved for
$s=2m$.
\end{proof}

\section{The Steklov zeta function $\zeta_{\alpha_\tau}$ and its first and second variations with respect to $\tau$}

We again consider a deformation
$\alpha_\tau$
of a positive function
$a\in C^\infty(\S)$
satisfying hypotheses of Lemma \ref{L2.2} with some
$l\in\N\cup\{\infty\}$.
Hypotheses of all statements in this section coincide with that of Lemma \ref{L2.2}. The hypotheses are not written explicitly for brevity.

\subsection{The resolvent operator on the positive semi-axis}

For
$\tau \in (-\ep,\ep)$,
let
$P_{0,\tau}:L^2({\S})\rightarrow L^2({\S})$
be the orthogonal projection onto the one-dimensional subspace spanned by
$\phi_{0,\tau}=(2\pi)^{-{1/2}}\alpha_\tau^{-{1/2}}$.
Then
$P_{0,\tau}$
is the
$C^l$ $\tau$-smooth family of projectors. When
$l\ge 1$
its derivative with respect to
$\tau$
is expressed by
\begin{equation}
{\pa P_{0,\tau}\over \pa \tau}=-{1\over 2}{\pa\ln(\alpha_\tau)\over \pa\tau}P_{0,\tau}-{1\over 2}P_{0,\tau}{\pa\ln(\alpha_\tau)\over \pa\tau}.
                                     \label{3.1}
\end{equation}
In particular,
$P_{0,\tau}\in C^l((-\ep,\ep),\mathcal{L}(H^s(\S),H^{s+1}(\S))$
for any
$s\in \R$.

When
$l\ge 1$
the derivatives
${\pa \Lambda_{\alpha_{\tau}}^2\over \pa \tau}$
and
${\pa D_{\alpha_{\tau}}^2\over \pa \tau}$
are pseudodifferential operators of order 2 whose symbols are degenerate. Lemma \ref{L2.2}, together with the equality
$$
{\pa \Lambda_{\alpha_{\tau}}^2\over \pa \tau}=\Lambda_{\alpha_{\tau}}{\pa \Lambda_{\alpha_{\tau}}\over \pa \tau}
+{\pa \Lambda_{\alpha_{\tau}}\over \pa \tau}\Lambda_{\alpha_{\tau}},
$$
implies that
$\Lambda_{\alpha_{\tau}}^2\in  C^l((-\ep,\ep),\mathcal{L}(H^{s+2}(\S),H^s(\S))$
for any
$s\in \R$.
Similarly
$D_{\alpha_{\tau}}^2\in  C^l((-\ep,\ep),\mathcal{L}(H^{s+2}(\S),H^s(\S))$.

Observe that
$\Lambda_{\alpha_{\tau}}^2$
is a non-negative self-adjoint operator whose kernel coincides with the one-dimensional space spanned by
$\phi_{0,\tau}$.
Therefore, for every
$\lambda\ge0$,
the operator
$\Lambda_{\alpha_{\tau}}^2+P_{0,\tau}+\lambda$
has the bounded inverse.
We consider the family of positive bounded operators
\begin{equation}
G(\tau,\lambda)=(\Lambda_{\alpha_{\tau}}^2+P_{0,\tau}+\lambda)^{-1},\quad (\tau,\lambda)\in (-\ep,\ep)\times [0,\infty).
                                                          \label{3.2}
\end{equation}

We start with the following statement.

\begin{lemma} \label{L3.1}
For every
$s\in\R$
and every
$\tau\in(-\ep,\ep)$,
the operator
$G(\tau,\lambda)$
belongs to
$\mathcal{L}(H^s(\S),H^{s+2}(\S))$
and the function
$$
\tau\mapsto\Big(\lambda\mapsto G(\tau,\lambda)\Big)
$$
belongs to
$C^l\big( (-\ep,\ep), C^\infty([0,+\infty), \mathcal{L}(H^s(\S),H^{s+2}(\S)))\big)$.
Moreover,
\begin{equation}
{\pa G\over \pa \lambda}(\tau,\lambda)=-\big(G(\tau,\lambda)\big)^2,\quad
{\pa G\over \pa \tau}(\tau,\lambda)=-G(\tau,\lambda) {\pa (\Lambda_{\alpha_{\tau}}^2+P_{0,\tau})\over \pa \tau}G(\tau,\lambda)\quad \mbox{\rm if}\ l\ge1.
                                                          \label{3.3}
\end{equation}
In addition, for every $\ep'\in(0,\ep)$,
\begin{equation}
\sup_{(\tau,\lambda)\in[-\ep',\ep']\times [0,\infty)}\|G(\tau,\lambda)\|_{\mathcal{L}(H^s(\S),H^{s+2}(\S))}<\infty
                                  \label{3.4}
\end{equation}
and
\begin{eqnarray}
 \sup_{(\tau,\lambda)\in[-\ep',\ep']\times [0,\infty)}(1+\lambda)\Big\|{\pa^{m_1+m_2} G\over \pa \tau^{m_1}\pa\lambda^{m_2}}(\tau,\lambda)\Big\|_{\mathcal{L}(H^s(\S),H^{s+2m_2}(\S))}<\infty
                                    \label{3.5}
\end{eqnarray}
for all
$l\ge m_1\in\N$
and
$m_2\in \N$.
\end{lemma}

\begin{proof}

The proof of the smoothness and of identities \eqref{3.3} repeats essentially the arguments used in the proof of Lemma \ref{L2.2} (here we deal with the operator
$\Lambda_{\alpha_\tau}^2+P_{0,\tau}\in \mathcal{L}(H^{s+2}(\S),H^s(\S))$
instead of the operator
$\Lambda_{\alpha_\tau}\in \mathcal{L}(H^{s+1}(\S),H^s(\S))$
in Lemma \ref{L2.2}).

It remains to prove \eqref{3.4}--\eqref{3.5}. We start with the following estimate: For every
$\ep'\in(0,\ep)$
and every real
$s$,
\begin{equation}
\sup_{\tau\in[-\ep',\ep']}\|(D^2+1)^s(\Lambda_{\alpha_{\tau}}^2+P_{0,\tau})^{-s}\|_{\mathcal{L}(L^2(\S))}
=\sup_{\tau\in[-\ep',\ep']}\|(\Lambda_{\alpha_{\tau}}^2+P_{0,\tau})^{-s}(D^2+1)^{s}\|_{\mathcal{L}(L^2(\S))}
<\infty.
                                                        \label{3.6}
\end{equation}
The estimate \eqref{3.6} follows from the
$\tau$-continuity of the family
$(\Lambda_{\alpha_{\tau}}^2+P_{0,\tau})^s\in \mathcal{L}(H^{2s}(\S), L^2(\S))$
when
$s$
is an integer. Then it is obtained for any real
$s$
by interpolation theory. (The
$\tau$-continuity is granted when
$s$
is an integer by composing
$(\Lambda_{\alpha_{\tau}}^2+P_{0,\tau})^{\pm 1}$.)

We use the identity
$$
(\Lambda_{\alpha_{\tau}}^2+P_{0,\tau})^{s+1} G(\tau,\lambda) (\Lambda_{\alpha_{\tau}}^2+P_{0,\tau})^{-s}=(\Lambda_{\alpha_{\tau}}^2+P_{0,\tau}) G(\tau,\lambda)
$$
to obtain
$$
\begin{aligned}
\|(\Lambda_{\alpha_{\tau}}^2+P_{0,\tau})^{s+1} G(\tau,\lambda) (\Lambda_{\alpha_{\tau}}^2+P_{0,\tau})^{-s}\|_{\mathcal{L}(L^2(\S))}
&=\|(\Lambda_{\alpha_{\tau}}^2+P_{0,\tau}) G(\tau,\lambda)\|_{\mathcal{L}(L^2(\S))}\\
&=\sup_{k\ge1}{\lambda_{k,\tau}^2\over \lambda_{k,\tau}^2+\lambda}= 1.
\end{aligned}
$$
Thus,
\begin{equation}
\sup_{(\tau,\lambda)\in(-\ep,\ep)\times [0,\infty)}\|(\Lambda_{\alpha_{\tau}}^2+P_{0,\tau})^{s+1} G(\tau,\lambda) (\Lambda_{\alpha_{\tau}}^2+P_{0,\tau})^{-s}\|_{\mathcal{L}(L^2(\S))}= 1.
                                                        \label{3.7}
\end{equation}

By the definition of
$H^s({\S})$-norms,
$$
\|G(\tau,\lambda)\|_{\mathcal{L}(H^s(\S),H^{s+2}(\S))}=\|(D^2+1)^{s/2+1}G(\tau,\lambda)(D^2+1)^{-s/2}\|_{{\mathcal L}(L^2({\S}))}.
$$
This can be written in the form
$$
\begin{aligned}
\|G(\tau,\lambda)\|_{\mathcal{L}(H^s(\S),H^{s+2}(\S))}
=\|&(D^2+1)^{s/2+1}(\Lambda_{\alpha_\tau}^2+P_{0,\tau})^{-s/2-1}\times\\
&\times(\Lambda_{\alpha_\tau}^2+P_{0,\tau})^{s/2+1}G(\tau,\lambda)
(\Lambda_{\alpha_\tau}^2+P_{0,\tau})^{-s/2}\times\\
&\times(\Lambda_{\alpha_\tau}^2+\!P_{0,\tau})^{s/2}(D^2+1)^{-s/2}\|_{{\mathcal L}(L^2({\S}))}
\end{aligned}
$$
and implies the inequality
$$
\begin{aligned}
\|G(\tau,\lambda)\|_{\mathcal{L}(H^s(\S),H^{s+2}(\S))}&\le
\|(D^2+1)^{s/2+1}(\Lambda_{\alpha_\tau}^2+P_{0,\tau})^{-s/2-1}\|_{{\mathcal L}(L^2({\S}))}\times\\
&\times\|(\Lambda_{\alpha_\tau}^2+P_{0,\tau})^{s/2+1}G(\tau,\lambda)
(\Lambda_{\alpha_\tau}^2+P_{0,\tau})^{-s/2}\|_{{\mathcal L}(L^2({\S}))}\times\\
&\times\|(\Lambda_{\alpha_\tau}^2+P_{0,\tau})^{s/2}(D^2+1)^{-s/2}\|_{{\mathcal L}(L^2({\S}))}.
\end{aligned}
$$
By \eqref{3.6}--\eqref{3.7}, three factors on the right-hand side of the inequality are bounded uniformly in
$(\tau,\lambda)\in[-\ep',\ep']\times [0,\infty)$.
This proves \eqref{3.4}.

Since the first positive Steklov eigenvalue $\lambda_{1,\tau}$ depends continuously on 
$\tau$,
there exists a positive constant 
$c_0<1$ 
such that
$$
\inf_{\tau\in [-\ep',\ep']}\lambda_{1,\tau}\ge c_0.
$$
Obviously,
$$
(1+\lambda)\|(\Lambda_{\alpha_{\tau}}^2+P_{0,\tau})^s G(\tau,\lambda) (\Lambda_{\alpha_{\tau}}^2+P_{0,\tau})^{-s}\|_{\mathcal{L}(L^2(\S))}\\
=\sup_{k\ge 1}\,\frac{1+\lambda}{\lambda_{k,\tau}^2+\lambda}
\le \frac{1+\lambda}{\lambda_{1,\tau}^2+\lambda}\le\frac{1}{c_0}.
$$
We obtain
\begin{equation}
\sup_{(\tau,\lambda)\in[-\ep',\ep']\times [0,\infty)}(1+\lambda)\|(\Lambda_{\alpha_{\tau}}^2+P_{0,\tau})^s G(\tau,\lambda) (\Lambda_{\alpha_{\tau}}^2+P_{0,\tau})^{-s}\|_{\mathcal{L}(L^2(\S))}\le c_0^{-1}.
                                      \label{3.8}
\end{equation}
In the same way as \eqref{3.4} was derived from \eqref{3.7}, we derive from \eqref{3.8} with the help of \eqref{3.6}
$$
\sup_{(\tau,\lambda)\in[-\ep',\ep']\times [0,\infty)}(1+\lambda)\|G(\tau,\lambda)\|_{\mathcal{L}(H^s(\S),H^s(\S))}<\infty.
$$
We have thus proved \eqref{3.5} in the case of
$m_1=m_2=0$.

For every integer
$0\le k\le l$, the estimates
\begin{equation}
\sup_{(\tau,\lambda)\in[-\ep',\ep']\times[0,\infty)}\left\|{\pa^k (\Lambda_{\alpha_{\tau}}^2+P_{0,\tau})\over \pa \tau^k}G(\tau,\lambda)\right\|_{\mathcal{L}(H^s(\S),H^s(\S))}<\infty,
                                      \label{3.9}
\end{equation}
\begin{equation}
\sup_{(\tau,\lambda)\in[-\ep',\ep']\times[0,\infty)}\left\| G(\tau,\lambda) {\pa^k (\Lambda_{\alpha_{\tau}}^2+P_{0,\tau})\over \pa \tau^k}\right\|_{\mathcal{L}(H^{-s}(\S),H^{-s}(\S))}<\infty
                                      \label{3.10}
\end{equation}
follow from \eqref{3.4} taking the
$C^l$ $\tau$-smoothness of
$(\Lambda_{\alpha_{\tau}}^2+P_{0,\tau})\in \mathcal{L}(H^{s+2}(\S),H^s(\S))$
into account.

Differentiating formulas \eqref{3.3}, we obtain recurrent relations that express
${\pa^{m_1+m_2} G\over \pa \tau^{m_1}\pa\lambda^{m_2}}$
through lower order derivatives
$$
{\pa^{m'_1+m'_2} G\over \pa \tau^{m'_1}\pa\lambda^{m'_2}}\ (m'_1+m'_2<m_1+m_2,\ m_1'<m_1)\quad\mbox{and}\quad {\pa^k(\Lambda_{\alpha_{\tau}}^2+P_{0,\tau})\over \pa \tau^k}\ (k\le m_1).
$$
With the help of \eqref{3.9}--\eqref{3.10}, the recurrent relations imply the validity of \eqref{3.5} inductively in
$m_1+m_2$.
\end{proof}

The family of positive bounded operators
\begin{equation}
G_0(\tau,\lambda)=(D_{\alpha_{\tau}}^2+P_{0,\tau}+\lambda)^{-1}\quad \big(\tau\in (-\ep,\ep),\ \lambda\in [0,\infty)\big)
                                      \label{3.11}
\end{equation}
also satisfies the corresponding statements of  Lemma \ref{L3.1} with appropriate changes.

For any
$(\tau,\lambda)\in (-\ep,\ep)\times [0,+\infty)$
the operator
$(G-G_0)(\tau,\lambda)$
is smoothing as well as all its derivatives.
More precisely we have the following property.

\begin{lemma}     \label{L3.2}
For every
$s\in \R$
and every
$m \in\N$,
the function
$
\tau\mapsto\big(\lambda\mapsto(G-G_0)(\tau,\lambda)\big)
$
belongs to
$
C^l\big((-\ep,\ep), C^\infty([0,+\infty),\mathcal{L}(H^s(\S),H^{s+m}(\S)))\big).
$

For every
$\ep'\in (0,\ep)$
and every
$(m_1,m_2,j_1,j_2)\in \N^4$ such that
$m_1\le l$,
\begin{equation}
\sup_{(\tau,\lambda)\in [-\ep',\ep']\times (0,+\infty)}(1+\lambda)^2\left\|D^{j_1}{\pa^{m_1+m_2}(G-G_0)\over \pa \tau^{m_1}\pa \lambda^{m_2}}(\tau,\lambda)D^{j_2}\right\|_{\mathcal{L}(L^2(\S))}<\infty.
                                               \label{3.12}
\end{equation}

\end{lemma}

\begin{proof}
We have the following identity:
\begin{equation}
(G-G_0)(\tau,\lambda)=-G(\tau,\lambda)(\Lambda_{\alpha_\tau}^2-D_{\alpha_\tau}^2)G_0(\tau,\lambda).
                                \label{3.13}
\end{equation}
We know that
$\Lambda_{\alpha_\tau}^2-D_{\alpha_\tau}^2$
is a smoothing operator with the Schwartz kernel belonging to
$C^l((-\ep,\ep), C^\infty(\S\times \S))$,
see \eqref{2.1}. Therefore
\begin{equation}
\Lambda_{\alpha_\tau}^2-D_{\alpha_\tau}^2\in C^l((-\ep,\ep),\mathcal{L}(H^{s'}(\S),H^{s'+m'}(\S)))
                                   \label{3.14}
\end{equation}
for any $s'\in\R$ and any $m'\in \N$.

Let $s\in \R$ and $m\in \N$. From Lemma \ref{L3.1} and its analog for $G_0$, we know that
$$
G_0\in C^l((-\ep,\ep),C^\infty([0,+\infty),\mathcal{L}(H^s(\S),H^{s+2}(\S)))),
$$
$$
G\in C^l((-\ep,\ep),C^\infty([0,+\infty),\mathcal{L}(H^{s+m-2}(\S),H^{s+m}(\S)))).
$$
The first statement of the lemma follows now from \eqref{3.13} and \eqref{3.14} for
$(s',m')=(s+2,m-4)$.

Inequality \eqref{3.12} also follows from \eqref{3.13}, \eqref{3.14}, \eqref{3.5} and the analog of \eqref{3.5} for
$G_0$.
\end{proof}

\subsection{Complex powers $(\Lambda_{\alpha_\tau}+P_{0,\tau})^{-z}$ for $\Re z\in (0,2)$}

We use the following definition of complex powers of a positive self-adjoint  operator
$A:L^2(\S)\rightarrow L^2(\S)$
with a discrete eigenvalue spectrum (see for example \cite{Kom}):
If
$\{\varphi_k\}_{k\in \N}$
is an orthonormal basis of
$L^2(\S)$
consisting of eigenvectors of
$A$
with associated eigenvalues
$\lambda_k>0$,
then
$$
A^{z}f=\sum_{k\in\N}\lambda_k^z\l f,\varphi_k \r\varphi_k\quad\mbox{for}\ z\in\C\ \mbox{and}\ f\in C^\infty(\S),
$$
where
$\lambda_k^z=e^{z\ln(\lambda_k)}$
and
$\ln(\lambda_k)\in\R$.
The series converges at least for
$\Re z\le0$.

By Lemma \ref{L3.1} and  \eqref{3.5} for $k=0$, we can consider the
$(\tau,z)$-continuous family of bounded positive operators
$$
\int_0^{+\infty}\lambda^{-z}G(\tau,\lambda) \,d\lambda,\quad
\int_0^{+\infty}\lambda^{-z}G_0(\tau,\lambda) \,d\lambda
$$
for any $z\in \C$ satisfying $0<\Re z<1$. One can use a basis of eigenvectors of $\Lambda_{\alpha_{\tau}}$ to show that
\begin{equation}
(\Lambda_{\alpha_{\tau}}+P_{0,\tau})^{-2z}=\gamma(z)\int_0^{\infty}\lambda^{-z}G(\tau,\lambda)\,d\lambda\quad\mbox{for}\quad0<\Re z<1,
                             \label{3.15}
\end{equation}
where
\begin{equation}
\gamma(z)=\Big(\int_0^{\infty}\lambda^{-z} (1+\lambda)^{-1}\,d\lambda\Big)^{-1}
=\frac{\sin\pi z}{\pi}\quad\mbox{for}\quad0<\Re z<1.
                                   \label{3.16}
\end{equation}
The second equality in \eqref{3.16} follows from Euler's integral. Indeed, changing the integration variable as 
$x=(1+\lambda)^{-1}$, one easily derives
$$
\int_0^{\infty}\lambda^{-z} (1+\lambda)^{-1}\,d\lambda
=\int_0^1x^{z-1}(1-x)^{-z}\,dx=B(z,1-z)=\frac{\pi}{\sin\pi z},
$$
where
$B$
is Euler's Beta-function.

Similarly
\begin{equation}
(|D_{\alpha_{\tau}}|+P_{0,\tau})^{-2z}=\gamma(z)\int_0^{\infty}\lambda^{-z}G_0(\tau,\lambda) \,d\lambda\quad\mbox{for}\quad0<\Re z<1.
                                   \label{3.17}
\end{equation}
With the help of Lemma \ref{L3.1}, we derive the following results.

\begin{lemma}   \label{L3.3}
The family of operators
$(\Lambda_{\alpha_{\tau}}+P_{0,\tau})^{-2z}$
belongs to
$$
C^l\big((-\ep,\ep),C^\infty(\{z\in \C\mid \Re z\in (0,1)\},\mathcal{L}(L^2(\S)))\big)
$$
and its first derivative with respect to
$z$
is given by
$$
{\pa\over \pa z}(\Lambda_{\alpha_{\tau}}+P_{0,\tau})^{-2z}=(\Lambda_{\alpha_{\tau}}+P_{0,\tau})^{-2z}{d\over dz}\ln(\gamma(z))-\gamma(z)\int_0^{\infty}\lambda^{-z}\ln(\lambda)\,G(\tau,\lambda) \,d\lambda.
$$
In the case of
$l\ge 1$,
the first derivatives with respect to
$\tau$
is given by
\begin{equation}
{\pa\over \pa\tau}(\Lambda_{\alpha_{\tau}}+P_{0,\tau})^{-2z}=-\gamma(z)\int_0^{\infty}\lambda^{-z}G(\tau,\lambda) {\pa(\Lambda_{\alpha_{\tau}}^2+P_{0,\tau})\over \pa \tau}G(\tau,\lambda)\,d\lambda.
                              \label{3.18}
\end{equation}

Similarly, the family  of bounded operators
$(|D_{\alpha_{\tau}}|+P_{0,\tau})^{-2z}$ belongs to
$$
C^l\big((-\ep,\ep),C^\infty(\{z\in \C\mid \Re z\in (0,1)\},\mathcal{L}(L^2(\S)))\big)
$$
and its first derivatives with respect to
$z$
is given by
$$
{\pa\over \pa z}(|D_{\alpha_{\tau}}|+P_{0,\tau})^{-2z}=(|D_{\alpha_{\tau}}|+P_{0,\tau})^{-2z}{d\over dz}\ln(\gamma(z))-\gamma(z)\int_0^{\infty}\lambda^{-z}\ln(\lambda)\,G_0(\tau,\lambda) \,d\lambda.
$$
In the case of
$l\ge 1$,
the first derivatives with respect to
$\tau$
is given by
\begin{equation}
{\pa\over \pa\tau}(|D_{\alpha_{\tau}}|+P_{0,\tau})^{-2z}=-\gamma(z)\int_0^{\infty}\lambda^{-z}G_0(\tau,\lambda) {\pa( D_{\alpha_{\tau}}^2+P_{0,\tau})\over \pa \tau}G_0(\tau,\lambda)\,d\lambda.
                            \label{3.19}
\end{equation}
\end{lemma}


\subsection{The family of smoothing operators $(\Lambda_{\alpha_\tau}+P_{0,\tau})^z-(|D_{\alpha_\tau}|+P_{0,\tau})^z$.}
For
$(\tau,z)\in (-\ep,\ep)\times \C$
let us denote by
$H(\tau,z)$
the operator from
$C^\infty(\S)$
to
$C^\infty(\S)$
defined by
\begin{equation}
H(\tau,z)=(\Lambda_{\alpha_\tau}+P_{0,\tau})^{-z}-(|D_{\alpha_\tau}|+P_{0,\tau})^{-z}.
                           \label{3.20}
\end{equation}
It is extended as a bounded operator on
$L^2(\S)$
when
$\Re z\ge 0$.
In the case of
$\Re z\ge -m\ (m\in \N)$,
it is extended as an operator from
$H^m(\S)$
to
$L^2(\S)$.
Actually we can improve the latter statement.

\begin{lemma}   \label{L3.4}
For every
$s\in \R$
and every
$m\in \N$,
\begin{equation}
H\in C^l\big((-\ep,\ep),C^\infty(\C,\mathcal{L}(H^s(\S),H^{s+m}(\S)))\big).
                           \label{3.21}
\end{equation}
For every compact
$K\subset\C$,
every
$\ep'\in(0,\ep)$
and every
$(j_1,j_2, m_1,m_2)\in \N^4$
such that
$m_1\le l$,
\begin{equation}
\sup_{(\tau,z)\in [-\ep',\ep']\times K}\Big\|D^{j_1}{\pa^{m_1+m_2}H\over \pa \tau^{m_1}\pa z^{m_2}}(\tau,z)D^{j_2}\Big\|_{\mathcal{L}(L^2(\S))}<\infty.
                                       \label{3.22}
\end{equation}
\end{lemma}

\begin{proof}
We start with the case when
$\Re z\in (0,2)$.
As is seen from \eqref{3.15} and \eqref{3.16},
$$
H(\tau,z)=\gamma(z/2)\int_0^{\infty}\lambda^{-{z/2}}(G-G_0)(\tau,\lambda)\,d\lambda.
$$
With the help of Lemma \ref{L3.2}, this implies
\begin{equation}
H\in C^l\big((-\ep,\ep),C^\infty( \{z\in \C\mid \Re z\in (0,2)\},\mathcal{L}(H^s(\S),H^{s+m}(\S)))\big)\quad\mbox{for}\quad (s,m)\in \R\times \N.
                                            \label{3.23}
\end{equation}

By \eqref{3.20},
$$
\begin{aligned}
(\Lambda_{\alpha_\tau}+P_{0,\tau})^2H(\tau,1)
&=(\Lambda_{\alpha_\tau}+P_{0,\tau})^2\Big((\Lambda_{\alpha_\tau}+P_{0,\tau})^{-1}-(|D_{\alpha_\tau}|+P_{0,\tau})^{-1}\Big)\\
&=(\Lambda_{\alpha_\tau}+P_{0,\tau})-(\Lambda_{\alpha_\tau}+P_{0,\tau})^2(|D_{\alpha_\tau}|+P_{0,\tau})^{-1}.
\end{aligned}
$$
We rewrite this in the form
$$
\begin{aligned}
(\Lambda_{\alpha_\tau}+P_{0,\tau})^2H(\tau,1)
&=\Big((\Lambda_{\alpha_\tau}+P_{0,\tau})-(|D_{\alpha_\tau}|+P_{0,\tau})\Big)\\
&-\Big((\Lambda_{\alpha_\tau}+P_{0,\tau})^2(|D_{\alpha_\tau}|+P_{0,\tau})^{-1}-(|D_{\alpha_\tau}|+P_{0,\tau})\Big)
\end{aligned}
$$
and again use \eqref{3.20} to obtain
\begin{equation}
(\Lambda_{\alpha_\tau}+P_{0,\tau})^2H(\tau,1)=H(\tau,-1)
-\Big((\Lambda_{\alpha_\tau}+P_{0,\tau})^2(|D_{\alpha_\tau}|+P_{0,\tau})^{-1}-(|D_{\alpha_\tau}|+P_{0,\tau})\Big).
                                            \label{3.24}
\end{equation}

On using the equalities
$$
(\Lambda_{\alpha_\tau}+P_{0,\tau})^2=\Lambda_{\alpha_\tau}^2+P_{0,\tau},\quad
|D_{\alpha_\tau}|+P_{0,\tau}=(D_{\alpha_\tau}^2+P_{0,\tau})^{1/2},
$$
we transform the second term on the right-hand side of \eqref{3.24} as follows:
$$
\begin{aligned}
(\Lambda_{\alpha_\tau}+P_{0,\tau})^2(|D_{\alpha_\tau}|+P_{0,\tau})^{-1}&-(|D_{\alpha_\tau}|+P_{0,\tau})=\\
&=(\Lambda_{\alpha_\tau}^2+P_{0,\tau})(D_{\alpha_\tau}^2+P_{0,\tau})^{-1/2}-(D_{\alpha_\tau}^2+P_{0,\tau})^{1/2}\\
&=\Big((\Lambda_{\alpha_\tau}^2+P_{0,\tau})-(D_{\alpha_\tau}^2+P_{0,\tau})\Big)(D_{\alpha_\tau}^2+P_{0,\tau})^{-1/2}\\
&=(\Lambda_{\alpha_\tau}^2-D_{\alpha_\tau}^2)(D_{\alpha_\tau}^2+P_{0,\tau})^{-1/2}.
\end{aligned}
$$
Substitute this value into \eqref{3.24}
\begin{equation}
(\Lambda_{\alpha_\tau}+P_{0,\tau})^2H(\tau,1)=H(\tau,-1)
-(\Lambda_{\alpha_\tau}^2-D_{\alpha_\tau}^2)(D_{\alpha_\tau}^2+P_{0,\tau})^{-1/2}.
                                            \label{3.25}
\end{equation}
By \eqref{3.11},
$(D_{\alpha_\tau}^2+P_{0,\tau})^{-1/2}=\big(G_0(\tau,0)\big)^{1/2}$,
and \eqref{3.25} takes the form
\begin{equation}
(\Lambda_{\alpha_\tau}+P_{0,\tau})^2H(\tau,1)=H(\tau,-1)
-(\Lambda_{\alpha_\tau}^2-D_{\alpha_\tau}^2)\big(G_0(\tau,0)\big)^{1/2}.
                                            \label{3.26}
\end{equation}

By \eqref{3.13},
$$
(G-G_0)(\tau,0)=-G(\tau,0)(\Lambda_{\alpha_\tau}^2-D_{\alpha_\tau}^2)G_0(\tau,0).
$$
Express $\Lambda_{\alpha_\tau}^2-D_{\alpha_\tau}^2$ from this equality
$$
\Lambda_{\alpha_\tau}^2-D_{\alpha_\tau}^2=-\big(G(\tau,0)\big)^{-1}(G-G_0)(\tau,0)\big(G_0(\tau,0)\big)^{-1}.
$$
Substitute this expression into \eqref{3.26}
\begin{equation}
(\Lambda_{\alpha_\tau}+P_{0,\tau})^2H(\tau,1)=H(\tau,-1)
+\big(G(\tau,0)\big)^{-1}(G-G_0)(\tau,0)\big(G_0(\tau,0)\big)^{-1/2}.
                                            \label{3.27}
\end{equation}

By \eqref{3.2},
$$
\big(G(\tau,0)\big)^{-1}=\Lambda_{\alpha_\tau}^2+P_{0,\tau}=(\Lambda_{\alpha_\tau}+P_{0,\tau})^2.
$$
Substituting this value into \eqref{3.27}, we finally obtain
\begin{equation}
(\Lambda_{\alpha_\tau}+P_{0,\tau})^2H(\tau,1)=H(\tau,-1)
+(\Lambda_{\alpha_\tau}+P_{0,\tau})^2(G-G_0)(\tau,0)\big(G_0(\tau,0)\big)^{-1/2}.
                                            \label{3.28}
\end{equation}

We write \eqref{3.28} in the form
\begin{equation}
H(\tau,-1)=(\Lambda_{\alpha_\tau}+P_{0,\tau})^2H(\tau,1)-(\Lambda_{\alpha_\tau}+P_{0,\tau})^2(G-G_0)(\tau,0)\big(G_0(\tau,0)\big)^{-1/2}.
                                            \label{3.29}
\end{equation}

With the help of Lemmas \ref{L2.2} and \ref{L3.2} and of \eqref{3.23} for
$z=1$,
\eqref{3.29} implies
\begin{equation}
H(\tau,-1)\in C^l\big((-\ep,\ep),\mathcal{L}(H^s(\S),H^{s+m}(\S))\big)\quad\mbox{for any}\quad (s,m)\in \R\times \N.
                                \label{3.30}
\end{equation}

Now, we prove by induction on
$k\in\N$
that
\begin{equation}
H\in C^l\big((-\ep,\ep), C^\infty(\{z\in \C\mid \Re z\in (-k,k+2)\},\mathcal{L}(H^s(\S),H^{s+m}(\S)))\big)\ \mbox{for}\ (s,m)\in \R\times \N.
                                      \label{3.31}
\end{equation}
For
$k=0$,
\eqref{3.31} coincides with \eqref{3.23}. Assume \eqref{3.31} to be valid for some
$0\le k\in\N$.

The recurrent relation
\begin{equation}
H(\tau,z)=(\Lambda_{\alpha_\tau}+P_{0,\tau})H(\tau,z+1)+H(\tau,-1)(|D_{\alpha_\tau}|+P_{0,\tau})^{-z-1}
                                        \label{3.32}
\end{equation}
easily follows from the definition \eqref{3.20}.

Together with Lemma \ref{L2.2}, the induction hypothesis \eqref{3.31} implies
that, for any $(s,m)\in\R\times\N$,
\begin{equation}
(\Lambda_{\alpha_\tau}\!+\!P_{0,\tau})H(\tau,z\!+\!1)
\in C^l\big((\!-\!\ep,\ep), C^\infty( \{z\in \C\mid \Re z\in (\!-\!k\!-\!1,k\!+\!1)\},\mathcal{L}(H^s(\S),H^{s\!+\!m}(\S)))\big).
                                      \label{3.33}
\end{equation}

The eigenbasis of
$|D_{\alpha_\tau}|+P_{0,\tau}$
is given by the family
$\{\phi_l\}_{l\in \Z}$,
see \eqref{2.2}. The eigenvalue associated to
$\phi_l$
is
$\max(|l|,1)$.
Therefore
$$
(|D_{\alpha_\tau}|+P_{0,\tau})^{-z-1}\in C^l\big((-\ep,\ep),C^\infty(\{z\in \C,\ \Re z\in (-\rho,\rho)\},\mathcal{L}(H^s(\S),H^{s+1-\rho}(\S)))\big)
$$
for any
$s\in \R$
and
$\rho>0$.
Together with \eqref{3.30}, this gives
\begin{equation}
H(\tau,-1)(|D_{\alpha_\tau}|+P_{0,\tau})^{-z-1}\in C^l\big((-\ep,\ep), C^\infty(\C,\mathcal{L}(H^s(\S),H^{s+m}(\S)))\big)\ \mbox{for}\ (s,m)\in \R\times \N.
                                \label{3.34}
\end{equation}

With the help of \eqref{3.33} and \eqref{3.34}, the recurrent relation \eqref{3.32} gives for any
$(s,m)\in \R\times \N$
$$
H\in C^l\big((-\ep,\ep),C^\infty( \{z\in \C,\ \Re z\in (-k-1,k+1)\},\mathcal{L}(H^s(\S),H^{s+m}(\S)))\big).
$$
Uniting this with \eqref{3.31}, we obtain for any
$(s,m)\in \R\times \N$
\begin{equation}
H\in C^l\big((-\ep,\ep), C^\infty(\{z\in \C\mid \Re z\in (-k-1,k+2)\},\mathcal{L}(H^s(\S),H^{s+m}(\S)))\big).
                                        \label{3.35}
\end{equation}

The recurrent relation
\begin{equation}
H(\tau,z)=(\Lambda_{\alpha_\tau}+P_{0,\tau})^{-1}H(\tau,z-1)+H(\tau,1)(|D_{\alpha_\tau}|+P_{0,\tau})^{-z+1}
                                       \label{3.36}
\end{equation}
is proved similarly to \eqref{3.32}. In the same way as \eqref{3.35} has been proven, the induction hypothesis \eqref{3.31} implies with the help of \eqref{3.36} that
\begin{equation}
H\in C^l\big((-\ep,\ep), C^\infty(\{z\in \C\mid \Re z\in (-k,k+3)\},\mathcal{L}(H^s(\S),H^{s+m}(\S)))\big)\ \mbox{for}\ (s,m)\in \R\times \N.
                                       \label{3.37}
\end{equation}

Uniting \eqref{3.35} and \eqref{3.37}, we obtain for any
$(s,m)\in \R\times \N$
$$
H\in C^l\big((-\ep,\ep), C^\infty(\{z\in \C\mid \Re z\in (-k-1,k+3)\},\mathcal{L}(H^s(\S),H^{s+m}(\S)))\big).
$$
This finishes the induction step.

Being valid for every
$k$,
\eqref{3.31} proves \eqref{3.21}.
\end{proof}

\subsection{Smoothness of $\zeta_{\alpha_\tau}(z)$ }
We recall that
$$
\zeta_{\alpha_\tau}(z)=2\zeta_R(z)+\T\big[H(\tau,z)\big].
$$

\begin{lemma}  \label{L3.5}
The function
$\zeta_{\alpha_\tau}(z)$
belongs to
$C^l\big( (-\ep,\ep), C^\infty(\C\b\{1\})\big)$
and, for
$l\ge 1$,
\begin{equation}
{\pa^{m_1+m_2}\zeta_{\alpha_\tau}(z)\over \pa \tau^{m_1}\pa z^{m_2}}=\T\Big[{\pa^{m_1+m_2}H\over \pa \tau^{m_1}\pa z^{m_2}}(\tau,z)\Big]
                                           \label{3.38}
\end{equation}
for any
$(\tau,z)\in (-\ep,\ep)\times (\C\b\{1\})$
and any
$m_1,m_2\in \N$
such that
$1\le m_1\le l$.
\end{lemma}

\begin{proof}
First we note that
$$
\zeta_{\alpha_\tau}(z)=2\zeta_R(z)+(2\pi)^{-1}\sum_{n\in \Z}g_n(\tau,z),
$$
where
$g_n(\tau,z)=\l H(\tau,z)e^{in\theta},e^{in\theta}\r$.
Lemma \ref{L3.4} implies that
$g_n\in C^l( (-\ep,\ep), C^\infty(\C\b\{1\}))$
and, for
$m_1\le l$,
$$
{\pa^{m_1+m_2}g_n\over \pa \tau^{m_1}\pa z^{m_2}}(\tau,z)
=\Big\langle{\pa^{m_1+m_2}H\over \pa \tau^{m_1}\pa z^{m_2}}(\tau,z) e^{in\theta},e^{in \theta}\Big\rangle.
$$
With the help of \eqref{3.22}, this implies for
$\ep'\in(0,\ep)$
and
$n\not=0$
$$
\left|{\pa^{m_1+m_2}g_n\over \pa \tau^{m_1}\pa z^{m_2}}(\tau,z)\right|
=n^{-2}\left|\Big\langle D^2{\pa^{m_1+m_2}H\over \pa \tau^{m_1}\pa z^{m_2}}(\tau,z) e^{in\theta},e^{in \theta}\Big\rangle\right|
\le Cn^{-2},
$$
where
$C=\sup_{\tau\in [-\ep',\ep']}\|D^2{\pa^{m1+m_2}H\over \pa \tau^{m_1}\pa z^{m_2}}(\tau,z)\|_{\mathcal{L}(L^2(\S))}<\infty$.
Using the classical fact on functions series, we see that
$\sum_{n\in \Z}g_n\in C^l\big( (-\ep,\ep), C^\infty(\C\b\{1\})\big)$.
This implies \eqref{3.38}.
\end{proof}

\subsection{First variation of $\zeta_{\alpha_\tau}(z)$ with respect to $\tau$}
We assume
$l\ge 1$
in this subsection.

\begin{lemma}    \label{L3.6}
For
$\tau\in (-\ep,\ep)$
and any
$z\in\C$,
\begin{equation}
{\pa \big(\zeta_{\alpha_\tau}(z)\big)\over \pa \tau}=-z\,\T\Big[H(\tau,z){\pa \ln(\alpha_\tau)\over \pa\tau}\Big].
                                        \label{3.39}
\end{equation}
At
$\tau=0$
it reads as
\begin{equation}
\left.{\pa \big(\zeta_{\alpha_\tau}(z)\big)\over \pa \tau}\right|_{\tau=0}=-z\,\T\big[\big((\Lambda_a+P_0)^{-z}-(|D_a|+P_0)^{-z}\big)a^{-1}\beta\big],
                                             \label{3.40}
\end{equation}
where
$\beta(\theta)=\frac{\partial\alpha}{\partial\tau}(\theta,0)$
is the direction of the variation
$\alpha_\tau$.
\end{lemma}

{\bf Remark.} The derivative
${\pa (\zeta_{\alpha_\tau}(z))\over \pa \tau}$
is well defined at any
$z\in\C$
although the zeta function
$\zeta_{\alpha_\tau}(z)$
is not defined at the pole
$z=1$.
See the remark after Theorem \ref{Th1.3}.

\begin{proof}
We reduce the computation to the case
$\Re z\in (0,2)$
by holomorphy in
$z\in \C$.
By \eqref{3.38},
\begin{equation}
{\pa \big(\zeta_{\alpha_\tau}(z)\big)\over \pa \tau }=\T\Big[{\pa H\over \pa \tau}(\tau,z)\Big].
                                  \label{3.41}
\end{equation}

Differentiate equality \eqref{3.20}
\begin{equation}
\frac{\partial H}{\partial\tau}=\frac{\partial (\Lambda_{\alpha_\tau}+P_{0,\tau})^{-z}}{\partial\tau}
-\frac{\partial (|D_{\alpha_\tau}|+P_{0,\tau})^{-z}}{\partial\tau}.
                           \label{3.42}
\end{equation}
By \eqref{3.18} and \eqref{3.19},
$$
\begin{aligned}
\frac{\partial (\Lambda_{\alpha_\tau}+P_{0,\tau})^{-z}}{\partial\tau}&=-\gamma(z/2)\int_0^\infty\lambda^{-z/2}
G(\tau,\lambda)\frac{\partial(\Lambda_{\alpha_\tau}^2+P_{0,\tau})}{\partial\tau}G(\tau,\lambda)\,d\lambda,\\
\frac{\partial (|D_{\alpha_\tau}|+P_{0,\tau})^{-z}}{\partial\tau}&=-\gamma(z/2)\int_0^\infty\lambda^{-z/2}
G_0(\tau,\lambda)\frac{\partial(D_{\alpha_\tau}^2+P_{0,\tau})}{\partial\tau}G_0(\tau,\lambda)\,d\lambda.
\end{aligned}
$$
Substitute these values into \eqref{3.42}
\begin{equation}
\frac{\partial H}{\partial\tau}(\tau,z)=
-\gamma(\frac{z}{2})\!\!\int_0^\infty\!\!\lambda^{-\frac{z}{2}}
\Big(G(\tau,\lambda)\frac{\partial(\Lambda_{\alpha_\tau}^2\!+\!P_{0,\tau})}{\partial\tau}G(\tau,\lambda)
-G_0(\tau,\lambda)\frac{\partial(D_{\alpha_\tau}^2\!+\!P_{0,\tau})}{\partial\tau}G_0(\tau,\lambda)\Big)d\lambda.
                           \label{3.43}
\end{equation}

Recall that, by \eqref{3.3},
\begin{equation}
\frac{\partial G}{\partial\tau}(\tau,\lambda)=-G(\tau,\lambda)\frac{\partial(\Lambda_{\alpha_\tau}^2\!+\!P_{0,\tau})}{\partial\tau}G(\tau,\lambda).
                           \label{3.44}
\end{equation}
The similar formula for $G_0$
$$
\frac{\partial G_0}{\partial\tau}(\tau,\lambda)=-G_0(\tau,\lambda)\frac{\partial(D_{\alpha_\tau}^2\!+\!P_{0,\tau})}{\partial\tau}G(\tau,\lambda)
$$
is proved in the same way as \eqref{3.3}. With the help of two last formulas, \eqref{3.43} takes the form
\begin{equation}
{\pa H\over \pa \tau}(\tau,z)=\gamma(z/2)\int_0^{\infty}\lambda^{-z/2}{\pa (G-G_0)\over \pa \tau}(\tau,\lambda)\,d\lambda,
                                  \label{3.45}
\end{equation}

By Lemma \ref{L3.1},
$G(\tau,\lambda)$
and
$G_0(\tau,\lambda)$
and their derivatives are trace class operators in
$L^2(\S)$
at fixed
$(\tau,\lambda)$
with an appropriate bound in
$\lambda$
given by \eqref{3.5}. Hence we can transpose the trace operator and integration over
$\lambda\in (0,+\infty)$
on \eqref{3.45}. In this way we obtain
\begin{equation}
\T\Big[{\pa H\over \pa \tau}\Big]
=\gamma(z/2)\int_0^{\infty}\lambda^{-z/2}\Big\{\T\Big[{\pa G\over \pa \tau}\Big]-\T\Big[{\pa G_0\over \pa \tau}\Big]\Big\}(\tau,\lambda)\,d\lambda.
                                  \label{3.46}
\end{equation}

Formula \eqref{3.44} implies
\begin{equation}
\T\Big[{\pa G\over \pa \tau}(\tau,\lambda)\Big]=-\T\Big[\big(G(\tau,\lambda)\big)^2\,{\pa(\Lambda_{\alpha_\tau}^2+P_{0,\tau})\over \pa\tau}\Big].
                                  \label{3.47}
\end{equation}
We have used the classical fact:
$\T[AB]=\T[BA]$
if
$A$
is a trace class operator and
$B$
is a bounded operator, see \cite[Theorem 3.1]{Si}.
As easily follows from \eqref{2.3} and \eqref{3.1},
\begin{eqnarray*}
{\pa(\Lambda_{\alpha_\tau}^2+P_{0,\tau})\over \pa\tau}&=&{1\over 2}{\pa\ln(\alpha_\tau)\over \pa\tau}\Lambda_{\alpha_\tau}^2
+{1\over 2}\Lambda_{\alpha_\tau}^2{\pa\ln(\alpha_\tau)\over \pa\tau}\\
&+&\Lambda_{\alpha_\tau}{\pa\ln(\alpha_\tau)\over \pa\tau}\Lambda_{\alpha_\tau}-{1\over 2}{\pa\ln(\alpha_\tau)\over \pa\tau}P_{0,\tau}-{1\over 2}P_{0,\tau}{\pa\ln(\alpha_\tau)\over \pa\tau}.
\end{eqnarray*}
We substitute this value into \eqref{3.47} and use again the classical property of the trace. Besides this, the operators
$\Lambda_{\alpha_\tau}$
and
$P_{0,\tau}$
commute with
$G(\tau,\lambda)$.
In this way we obtain
$$
\T\Big[{\pa G\over \pa \tau}(\tau,\lambda)\Big]=-2\T\Big[\big(G(\tau,\lambda)\big)^2\Lambda_{\alpha_\tau}^2{\pa\ln(\alpha_\tau)\over \pa\tau}\Big]
+\T\Big[\big(G(\tau,\lambda)\big)^2P_{0,\tau}{\pa\ln(\alpha_\tau)\over \pa\tau}\Big].
$$
On using the equality
$$
\Lambda_{\alpha_\tau}^2=\big(G(\tau,\lambda)\big)^{-1}-P_{0,\tau}-\lambda
$$
that follows from \eqref{3.2}, we transform the previous formula to the form
\begin{equation}
\begin{aligned}
\T\Big[{\pa G\over \pa \tau}(\tau,\lambda)\Big]&=-2\T\Big[G(\tau,\lambda){\pa\ln(\alpha_\tau)\over \pa\tau}\Big]
+2\lambda\T\Big[\big(G(\tau,\lambda)\big)^2P_{0,\tau}{\pa\ln(\alpha_\tau)\over \pa\tau}\Big]\\
&+3\T\Big[\big(G(\tau,\lambda)\big)^2P_{0,\tau}{\pa\ln(\alpha_\tau)\over \pa\tau}\Big].
\end{aligned}
                                  \label{3.48}
\end{equation}
As follows from \eqref{3.3},
$$
\begin{aligned}
\frac{\partial}{\partial\lambda}\Big\{-2\lambda\T\Big[G(\tau,\lambda)P_{0,\tau}{\pa\ln(\alpha_\tau)\over \pa\tau}\Big]\Big\}
&=-2\T\Big[G(\tau,\lambda){\pa\ln(\alpha_\tau)\over \pa\tau}\Big]\\
&+2\lambda\T\Big[\big(G(\tau,\lambda)\big)^2P_{0,\tau}{\pa\ln(\alpha_\tau)\over \pa\tau}\Big].
\end{aligned}
$$
Therefore formula \eqref{3.48} takes its final form
\begin{equation}
\T\Big[{\pa G\over \pa \tau}(\tau,\lambda)\Big]=
\frac{\partial}{\partial\lambda}\Big\{-2\lambda\T\Big[G(\tau,\lambda)P_{0,\tau}{\pa\ln(\alpha_\tau)\over \pa\tau}\Big]\Big\}
+3\T\Big[\big(G(\tau,\lambda)\big)^2P_{0,\tau}{\pa\ln(\alpha_\tau)\over \pa\tau}\Big].
                                  \label{3.49}
\end{equation}

The similar formula for
$G_0$
is obtained in the same way:
\begin{equation}
\T\Big[{\pa G_0\over \pa \tau}(\tau,\lambda)\Big]=
\frac{\partial}{\partial\lambda}\Big\{-2\lambda\T\Big[G_0(\tau,\lambda)P_{0,\tau}{\pa\ln(\alpha_\tau)\over \pa\tau}\Big]\Big\}
+3\T\Big[\big(G_0(\tau,\lambda)\big)^2P_{0,\tau}{\pa\ln(\alpha_\tau)\over \pa\tau}\Big].
                                  \label{3.50}
\end{equation}
Take the difference of equations \eqref{3.49} and \eqref{3.50}. Taking the equality
$$
\big(G_0(\tau,\lambda)\big)^2P_{0,\tau}=\big(G(\tau,\lambda)\big)^2P_{0,\tau}
$$
into account, we obtain
\begin{equation}
\T\Big[{\pa (G-G_0)\over \pa \tau}(\tau,\lambda)\Big]=
{\pa\over \pa\lambda}\Big\{-2\lambda\T\Big[(G-G_0)(\tau,\lambda){\pa\ln(\alpha_\tau)\over \pa\tau}\Big]\Big\}.
                                  \label{3.51}
\end{equation}

Next, we multiply equation \eqref{3.51} by
 $\lambda^{-z/2}$
 and integrate with respect to
 $\lambda$
$$
\int_0^\infty\lambda^{-z/2}\,\T\Big[{\pa (G-G_0)\over \pa \tau}(\tau,\lambda)\Big]\,d\lambda=
\int_0^\infty\lambda^{-z/2}\,{\pa\over \pa\lambda}\Big\{-2\lambda\T\Big[(G-G_0)(\tau,\lambda){\pa\ln(\alpha_\tau)\over \pa\tau}\Big]\Big\}\,d\lambda.
$$
We transform the right-hand side with the help of integration by parts. Since
$\Re z\in (0,2)$,
the integrated term is equal to zero by \eqref{3.12} (with
$m_1=m_2=0$). In this way we obtain
$$
\int_0^\infty\lambda^{-z/2}\,\T\Big[{\pa (G-G_0)\over \pa \tau}(\tau,\lambda)\Big]\,d\lambda=
-z\int_0^\infty\lambda^{-z/2}\,\T\Big[(G-G_0)(\tau,\lambda){\pa\ln(\alpha_\tau)\over \pa\tau}\Big]\,d\lambda.
$$
Comparing this with \eqref{3.46}, we see that
\begin{equation}
\T\Big[{\pa H\over \pa \tau}\Big]=-z\gamma(z/2)\int_0^\infty\lambda^{-z/2}\,\T\Big[(G-G_0)(\tau,\lambda){\pa\ln(\alpha_\tau)\over \pa\tau}\Big]\,d\lambda.
                                  \label{3.52}
\end{equation}

By \eqref{3.20},
$$
H(\tau,z)=(\Lambda_{\alpha_\tau}+P_{0,\tau})^{-z}-(|D_{\alpha_\tau}|+P_{0,\tau})^{-z}.
$$
By \eqref{3.15} and \eqref{3.17},
$$
(\Lambda_{\alpha_\tau}+P_{0,\tau})^{-z}=\gamma(z/2)\int_0^\infty\lambda^{-z/2}G(\tau,\lambda)\,d\lambda,
$$
$$
(|D_{\alpha_\tau}|+P_{0,\tau})^{-z}=\gamma(z/2)\int_0^\infty\lambda^{-z/2}G_0(\tau,\lambda)\,d\lambda.
$$
Three last formulas imply
$$
H(\tau,z)=\gamma(z/2)\int_0^\infty\lambda^{-z/2}(G-G_0)(\tau,\lambda)\,d\lambda.
$$
Multiply this equality from the right by the operator of multiplication by the function
${\pa\ln(\alpha_\tau)\over \pa\tau}$.
The operator can be moved inside the integral since it is independent of
$\lambda$.
In this way we obtain
$$
H(\tau,z){\pa\ln(\alpha_\tau)\over \pa\tau}=\gamma(z/2)\int_0^\infty\lambda^{-z/2}(G-G_0)(\tau,\lambda){\pa\ln(\alpha_\tau)\over \pa\tau}\,d\lambda.
$$
Take the trace of both part. Again, the trace operator can be moved inside the integral and we get
$$
\T\Big[H(\tau,z){\pa\ln(\alpha_\tau)\over \pa\tau}\Big]
=\gamma(z/2)\int_0^\infty\lambda^{-z/2}\,\T\Big[(G-G_0)(\tau,\lambda){\pa\ln(\alpha_\tau)\over \pa\tau}\Big]\,d\lambda.
$$
The comparison of this formula with \eqref{3.52} gives
$$
\T\Big[{\pa H\over \pa \tau}(\tau,z)\Big]=-z\T\Big[H(\tau,z){\pa\ln(\alpha_\tau)\over \pa\tau}\Big].
$$
Together with \eqref{3.41}, this gives \eqref{3.39}.
\end{proof}

\subsection{Second variation of $\zeta_{\alpha_\tau}(z)$ with respect to $\tau$}
We assume
$l\ge 2$
in this subsection.
Repeating arguments from the proof of Lemma \ref{L3.5}, we prove that the right-hand side of \eqref{3.39} belongs to
$C^{l-1}( (-\ep,\ep), C^\infty(\C))$.
Then, differentiating equation \eqref{3.39}, we obtain the following expression for the second derivative.

\begin{lemma} \label{L3.7}
For every
$z\in\C$,
\begin{equation}
\begin{aligned}
\left.{\pa^2\big(\zeta_{\alpha_\tau}(z)\big)\over \pa\tau^2}\right|_{\tau=0}
&=-z\T\Big[{\pa H\over \pa\tau}(0,z) a^{-1}\beta\Big]\\
&+z\T\big[ \big((\Lambda_a+P_0)^{-z}-(|D_a|+P_0)^{-z}\big)a^{-2} \beta^2\big]\\
&-z\T\Big[\big((\Lambda_a+P_0)^{-z}-(|D_a|+P_0)^{-z}\big) a^{-2} \left.{\pa^2 \alpha_\tau\over \pa\tau^2}\right|_{\tau=0}\Big],
\end{aligned}
                                       \label{3.53}
\end{equation}
where
$\beta(\theta)=\frac{\partial\alpha}{\partial\tau}(\theta,0)$
is the direction of the variation
$\alpha_\tau$.

In more generality, for a 2-parametric deformation
$\alpha_{\tau,s}$,
$$
\begin{aligned}
\left.{\pa^2\big(\zeta_{\alpha_{\tau,s}}(z)\big)\over \pa\tau \pa s}\right|_{(\tau,s)=(0,0)}
&=-z\T\Big[{\pa H\over \pa s}(0,0,z) a^{-1}\beta\Big]\\
&+z\T\big[ \big((\Lambda_a+P_0)^{-z}-(|D_a|+P_0)^{-z}\big)a^{-2} \beta \gamma\big]\\
&-z\T\Big[\big((\Lambda_a+P_0)^{-z}-(|D_a|+P_0)^{-z}\big) a^{-2}
\left.{\pa^2 \alpha_{\tau,s}\over \pa\tau \pa s}\right|_{(\tau,s)=(0,0)}\Big],
\end{aligned}
$$
where
$\gamma=\left.{\pa\alpha_{0,s}\over \pa s}\right|_{s=0}$,
$\beta=\left.{\pa\alpha_{\tau,0}\over \pa \tau}\right|_{\tau=0}$.
\end{lemma}

\subsection{Application: behavior near $a={\mathbf 1}$}

Hereafter,
$\{{\hat u}_k\}_{k\in\Z}$
are the Fourier coefficients of a function
$u\in L^2(\S)$,
i.e.,
$u=\sum_{k\in\Z}{\hat u}_k\,e^{ik\theta}$.

We have the following result.

\begin{proposition} \label{P3.1}
Let
$\alpha_\tau$
be a
$C^2$-smooth variation of the function
$a={\mathbf 1}$
(the function identically equal to 1). Then, for every
$z\in\C$,
\begin{equation}
\left.{\pa \big(\zeta_{\alpha_\tau}(z)\big)\over \pa \tau}\right|_{\tau=0}=0,
                                  \label{3.54}
\end{equation}
\begin{equation}
\left.{\pa^2 \big(\zeta_{\alpha_\tau}(z)\big)\over \pa \tau^2}\right|_{\tau=0}=
4z\sum_{(n,p)\in \N^2\atop p>0,\ n>0}{n^{-z}-p^{-z}\over p^2-n^2}\,pn\,|\hat\beta_{p+n}|^2+2z^2\sum_{n>0}|n|^{-z}\,|\hat\beta_{2n}|^2,
                                  \label{3.55}
\end{equation}
\end{proposition}
where
$\beta(\theta)=\left.\frac{\partial\alpha_\tau(\theta)}{\partial\tau}\right|_{\tau=0}$.

\begin{proof}
The first variation formula \eqref{3.40} gives \eqref{3.54}.
Indeed,
$\Lambda_a=|D_a|$
for
$a={\mathbf 1}$.

The second variation formula \eqref{3.53} gives
\begin{equation}
\left.{\pa^2 \big(\zeta_{\alpha_\tau}(z)\big)\over \pa \tau^2}\right|_{\tau=0}
=-z\T\Big[{\pa H\over \pa\tau}(0,z)\beta\Big].
                       \label{3.56}
\end{equation}

We use the trigonometric  basis
$\{(2\pi)^{-1/2}e^{in\theta}\}_{n\in \Z}$
to compute the trace
$\T\Big[{\pa H\over \pa\tau}(0,z)\beta\Big]$:
$$
\T\Big[{\pa H\over \pa\tau}(0,z) \beta\Big]=(2\pi)^{-1}\sum_{p\in  \Z}\Big\langle {\pa H\over \pa\tau}(0,z) \beta e^{ip\theta},e^{ip\theta}\Big\rangle.
$$
Substituting
$\beta=\sum_{n\in\Z}{\hat\beta}_ne^{in\theta}$,
we obtain
\begin{equation}
\T\Big[{\pa H\over \pa\tau}(0,z) \beta\Big]
=(2\pi)^{-1}\sum_{n,p\in\Z}\hat\beta_{n-p}\Big\langle{\pa H\over \pa\tau}(0,z) e^{in\theta},e^{ip\theta}\Big\rangle.
                       \label{3.57}
\end{equation}

We have thus to compute
$\Big\langle {\pa H\over \pa\tau}(0,z)e^{in\theta},e^{ip\theta}\rangle$.
We reduce the computation to the case
$\Re z \in (0,1)$
by holomorphy in the
$z$-variable.

With the help of the definition \eqref{3.20} of the operator
$H$,
Formulas \eqref{3.18} and \eqref{3.19} give
$$
{\pa H\over \pa\tau}(\tau,z)=-\gamma(\frac{z}{2})\!\int\limits_0^{\infty}\!\lambda^{-\frac{z}{2}}\Big(\!
G(\tau,\lambda) {\pa(\Lambda_{\alpha_{\tau}}^2\!+\!P_{0,\tau})\over \pa \tau}G(\tau,\lambda)
-G_0(\tau,\lambda)  {\pa( D_{\alpha_{\tau}}^2\!+\!P_{0,\tau})\over \pa \tau}G_0(\tau,\lambda)\!\Big)d\lambda.
$$
Setting
$\tau=0$
here and using the equalities
$$
G(0,\lambda)=G_0(0,\lambda)=(\Lambda^2+ P_{e_0}+\lambda)^{-1},
$$
we obtain
$$
{\pa H\over \pa\tau}(0,z)=-\gamma(z/2)\int_0^{\infty}\lambda^{-z/2}(\Lambda^2+ P_{e_0}+\lambda)^{-1}
\left.{\pa(\Lambda_{\alpha_{\tau}}^2-D_{\alpha_\tau}^2)\over \pa \tau}\right|_{\tau=0}(\Lambda^2+ P_{e_0}+\lambda)^{-1}\,d\lambda,
$$
where
$P_{e_0}$
is the orthogonal projection onto the line spanned by
$e_0={1\over \sqrt{2\pi}}{\mathbf 1}$.
Then
\begin{eqnarray*}
{\pa(\Lambda_{\alpha_{\tau}}^2-D_{\alpha_\tau}^2)\over \pa \tau}&=&{1\over 2}{\pa \ln(\alpha_\tau)\over \pa\tau}(\Lambda_{\alpha_\tau}^2-D_{\alpha_\tau}^2)
+{1\over 2}(\Lambda_{\alpha_\tau}^2-D_{\alpha_\tau}^2){\pa \ln(\alpha_\tau)\over \pa\tau}\\
&&+\Lambda_{\alpha_\tau}{\pa \ln(\alpha_\tau)\over \pa\tau}\Lambda_{\alpha_\tau}-D_{\alpha_\tau}{\pa \ln(\alpha_\tau)\over \pa\tau}D_{\alpha_\tau}.
\end{eqnarray*}
At $\tau=0$, this becomes
$$
\left.{\pa(\Lambda_{\alpha_{\tau}}^2-D_{\alpha_\tau}^2)\over \pa \tau}\right|_{\tau=0}=\Lambda \beta\Lambda-D \beta D.
$$
Hence
$$
{\pa H\over \pa\tau}(0,z)=-\gamma(z/2)\int_0^{\infty}\lambda^{-z/2}(\Lambda^2+P_{e_0}+\lambda)^{-1}
(\Lambda \beta\Lambda-D \beta D)(\Lambda^2+P_{e_0}+\lambda)^{-1}\,d\lambda.
$$

With the help of the last formula, we obtain
$$
\begin{aligned}
\Big\langle {\pa H\over \pa\tau}&(0,z)e^{in\theta},e^{ip\theta}\Big\rangle=\\
&=-\gamma(z/2)\int_0^{\infty}\lambda^{-z/2}\big\langle
(\Lambda \beta\Lambda-D \beta D)(\Lambda^2+P_{e_0}+\lambda)^{-1} e^{in\theta},(\Lambda^2+P_{e_0}+\lambda)^{-1}e^{ip\theta}\big\rangle\,d\lambda.
\end{aligned}
$$
After elementary calculations, this becomes
\begin{equation}
\Big\langle {\pa H\over \pa\tau}(0,z)e^{in\theta},e^{ip\theta}\Big\rangle=
\left\{\begin{array}{l}
0\ \textrm{if either}\ n=0\ \textrm{or}\ p=0,\\
[5pt]
-\gamma(z/2)\int_0^{\infty}\lambda^{-z/2}(p^2+\lambda)^{-1}(n^2+\lambda)^{-1}\,d\lambda\times\\
[5pt]
\quad\times 2\pi(|np|-np)\hat \beta_{p-n}\quad \textrm{otherwise}.
\end{array}\right.
                                       \label{3.58}
\end{equation}

For positive reals
$x$
and
$y\ (x\not= y)$,
$$
(x+\lambda)^{-1}(y+\lambda)^{-1}={1\over y-x}\big((x+\lambda)^{-1}-(y+\lambda)^{-1}\big).
$$
With the help of \eqref{3.16}, this gives
\begin{equation}
-\gamma(z/2)\int_0^{\infty}\lambda^{-z/2}(x+\lambda)^{-1}(y+\lambda)^{-1}\,d\lambda={x^{-z/2}-y^{-z/2}\over x-y}.
                                       \label{3.59}
\end{equation}
In particular, when
$x\to y$,
\begin{equation}
-\gamma(z/2)\int_0^{\infty}\lambda^{-z/2}(y+\lambda)^{-2}\,d\lambda=-\frac{1}{2}\,zy^{-z/2-1}.
                                      \label{3.60}
\end{equation}

Combining \eqref{3.58}--\eqref{3.60}, we obtain
\begin{equation}
\Big\langle {\pa H\over \pa\tau}(0,z)e^{in\theta},e^{ip\theta}\Big\rangle
=\left\lbrace
\begin{array}{l}
0\quad \textrm{if}\ np\ge0,\\
[5pt]
-4\pi\,np\,\frac{\textstyle |n|^{-z}-|p|^{-z}}{\textstyle n^2-p^2}\hat\beta_{p-n}\quad \textrm{if}\ np<0\ \textrm{and}\ n\not=-p,\\
[5pt]
-2\pi\,z|p|^{-z}\hat\beta_{2p}\quad \textrm{if}\ n=-p\not=0.
\end{array}
\right.
                                 \label{3.61}
\end{equation}
The formula is valid for all $z\in \C$ since right-hand sides are entire functions.

We substitute \eqref{3.61} into \eqref{3.57} and use that
$\hat\beta_k=\overline{\hat\beta_{-k}}$
($\beta$
is a real function)
$$
\T\Big[{\pa H\over \pa\tau}(0,z) \beta\Big]
=-2\sum\limits_{np<0,\ n\neq -p}np\,\frac{|n|^{-z}-|p|^{-z}}{n^2-p^2}\,|\hat\beta_{n-p}|^2
-z\sum\limits_{p\neq0}|p|^{-z}\,|\hat\beta_{2p}|^2.
$$
After the change
$p:=-p$
of the summation index in the first sum, this becomes
$$
\T\Big[{\pa H\over \pa\tau}(0,z) \beta\Big]
=4\sum\limits_{n>0,\ p>0}np\,\frac{|n|^{-z}-|p|^{-z}}{n^2-p^2}\,|\hat\beta_{n+p}|^2
-2z\sum\limits_{n>0}|n|^{-z}\,|\hat\beta_{2n}|^2.
$$
Finally, substituting this expression into \eqref{3.56}, we obtain \eqref{3.55}.
\end{proof}

\section{Proof of Theorem \ref{Th1.2}}

The proof of Theorem \ref{Th1.3} is postponed to Section 6. Here, assuming Theorem \ref{Th1.3} to be valid, we prove Theorem \ref{Th1.2}.
The proof of the theorem is based on the first variation formula applied to the deformation of Theorem \ref{Th1.3}. We start with some important preliminaries that, besides the proof of Theorem \ref{Th1.2}, will play a key role in the construction of the deformation of Theorem \ref{Th1.3}..

\subsection{An alternative form of the first variation formula.}
Let a deformation
$\alpha_\tau\ (-\ep<\tau<\ep)$
of a positive function
$a\in C^\infty(\S)$
satisfy hypotheses of Lemma \ref{L2.2} with
$l=\infty$.
Differentiating equation \eqref{1.7} with respect to
$\tau$,
we obtain
\begin{equation}
\int_{\S}{\pa\alpha_\tau^{-1}\over\pa\tau}\,d\theta=-\int_{\S}\alpha_\tau^{-2}\,{\pa\alpha_\tau\over\pa\tau}\,d\theta=0.
                                          \label{4.1}
\end{equation}
Let us define the family of functions
$g_\tau\in C^\infty((-\ep,\ep),C^\infty(\S))$
by
\begin{equation}
\alpha_\tau g_\tau'-g_\tau\alpha_\tau'={\pa\alpha_\tau\over\pa\tau},\quad \int_{\S}g_\tau=0.
                                       \label{4.2}
\end{equation}
In other words
$\big({g_\tau\over \alpha_\tau}\big)'=-{\pa\alpha_\tau^{-1}\over\pa\tau}$
and
$\int_{\S}g_\tau=0$.
Such a family exists and is unique due to \eqref{4.1}. We also denote
\begin{equation}
g=\left.{g_\tau}\right|_{\tau=0},\quad \beta=\left.{\pa\alpha_\tau\over \pa \tau}\right|_{\tau=0}=ag'-a'g.
                                       \label{4.3}
\end{equation}

Let
$P_0:L^2({\mathbb S})\rightarrow L^2({\mathbb S})$
be the orthogonal projection onto the one-dimensional subspace spanned by the vector
$\phi_0=(2\pi)^{-1/2}a^{-1/2}$
(compare with \eqref{2.2}). We emphasize that
$P_0$
depends on the function
$a$
although the dependance is not designated explicitly.

\begin{theorem} \label{Th4.1}
Given a deformation
$\alpha_\tau$
of a positive function
$a\in C^\infty(\S)$
satisfying hypotheses of Lemma \ref{L2.2} with
$l=\infty$,
let the function
$g\in C^\infty(\S)$
be defined by \eqref{4.3}. Then, for every
$z\in \C$,
\begin{equation}
\left.{\pa \big(\zeta_{\alpha_\tau}(z)\big)\over \pa \tau}\right|_{\tau=0}
=-iz\,\T\big[(\Lambda_a+P_0)^{-z+1}(I-P_0)a^{-1/2}[\H,g]a^{-1/2}\big],
                                          \label{4.4}
\end{equation}
where
$[\H,g]$
is the commutator of the Hilbert transform
$\H$
and the operator of multiplication by the function
$g$.
In the case when
$g=i\H a$,
the formula simplifies to the following one:
\begin{equation}
\left.{\pa \big(\zeta_{\alpha_\tau}(z)\big)\over \pa \tau}\right|_{\tau=0}=z\,\T\big[(\Lambda_a+P_0)^{-z-1}(\Lambda_a^2-D_a^2)\big].
                                  \label{4.5}
\end{equation}
\end{theorem}

See the remark after Theorem \ref{Th1.3} concerning the value of the derivative
${\pa \big(\zeta_{\alpha_\tau}(z)\big)\over \pa \tau}$
at the pole
$z=1$.
Right-hand sides of formulas \eqref{4.4} and \eqref{4.5} are entire functions of
$z$
since
$[\H,g]$
and
$\Lambda_a^2-D_a^2=a^{1/2}D[\H,a]\H Da^{1/2}$
are smoothing operators.

The proof of Theorem \ref{Th4.1} is postponed to Section \ref{S5}.

\subsection{The good sign.}
The example
$g=i\H a$
is a right choice since we have the following interesting properties.

\begin{lemma} \label{L4.1}
Let a deformation
$\alpha_\tau$
of a positive function
$a\in C^\infty(\S)$
satisfy hypotheses of Theorem \ref{Th4.1}. Then
\begin{equation}
\T\big[(\Lambda_a+P_0)^{s-1}(\Lambda_a^2- D_a^2)\big]\ge0\quad\textrm{when}\ s>0,
                               \label{4.6}
\end{equation}
\begin{equation}
\T\big[(\Lambda_a+P_0)^{s-1}(\Lambda_a^2- D_a^2)\big]\le0\quad\textrm{when}\ s< 0.
                                        \label{4.7}
\end{equation}
Additionally, the equality
$$
\T\big[(\Lambda_a+P_0)^{s-1}(\Lambda_a^2- D_a^2)\big]=0
$$
holds for some
$0\neq s\in\R$
if and only if
$a$
is conformally equivalent to the constant-valued function
${\mathbf 1}$.
\end{lemma}

The proof of the lemma is presented in Section 5.
Combining Theorem \ref{Th4.1} and Lemma \ref{L4.1}, we obtain the following result.

\begin{corollary} \label{C4.1}
Under hypotheses of Theorem \ref{Th4.1}, assume that
$g=i\H a$
(or equivalently
$\beta=-a(\Lambda a) +(\H a) (Da)$).
Then
\begin{equation}
\left.{\pa \big(\zeta_{\alpha_\tau}(s)\big)\over \pa \tau}\right|_{\tau=0}\le 0
                                    \label{4.8}
\end{equation}
for every real
$s$.
If the equality holds in \eqref{4.8} for some
$0\neq s\in\R$,
then
$a$
is conformally equivalent to the constant-valued function
$\mathbf 1$.
\end{corollary}

\subsection{Proof of Theorem \ref{Th1.2}}

We start with proving \eqref{1.6}. Let
$\alpha_\tau\ (0\le\tau<\infty)$
be the deformation from Theorem \ref{Th1.3}. Here
$\alpha_0=a$.
By  statement (3) of Theorem \ref{Th1.3},
$\zeta_{\alpha_\tau}(s)$
is smooth and non-increasing in
$\tau\in[0,\infty)$
for any real
$s\neq1$.
We would like to prove that
\begin{equation}
\inf_{\tau\in[0,\infty)}\zeta_{\alpha_\tau}(s)=2\zeta_R(s)\quad\mbox{for}\quad 1\neq s\in\R.
                                      \label{4.9}
\end{equation}
Let us consider
$\Gamma_\ep=\alpha_{1/\ep}\ (0<\ep<\infty)$.
By statement (4) of Theorem \ref{Th1.3},
$$
\Gamma_\ep\to {\mathbf 1}\quad \textrm{in}\quad C^\infty(\S)\quad{as}\quad\ep \to 0^+.
$$
Setting
$\Gamma_{-\ep}=\Gamma_\ep$
for
$\ep>0$
and
$\Gamma_0={\mathbf 1}$,
we have defined the continuous path
$$
\R\rightarrow C^\infty(\S),\quad \ep\mapsto\Gamma_\ep
$$
consisting of positive functions. For a fixed
$1\neq s\in\R$,
the function
$\ep\mapsto\zeta_{\Gamma_\ep}(s)$
is continuous on
$\R$.
Hence
$$
\inf_{\tau\in(0,\tau_0)}\zeta_{\alpha_\tau}(s)=\zeta_{\alpha_{\tau_0}}(s)
=\zeta_{\Gamma_{\tau_0^{-1}}}(s)\to \zeta_{\Gamma_0}(s)=2\zeta_R(s)\quad\textrm{as}\quad \tau_0\to\infty.
$$
This implies \eqref{4.9}. We have thus proved \eqref{1.6}.

Now assume that
$\zeta_a(s)-2\zeta_R(s)=0$
for some
$0\neq s\in\R$.
Since the function
$\zeta_{\alpha_\tau}(s)-2\zeta_R(s)$
is non-increasing in
$\tau$, we conclude that
$$
\zeta_{\alpha_\tau}(s)-2\zeta_R(s)=0\quad \textrm{for all}\quad\tau \in [0,\infty).
$$
In particular the derivative
${\pa \zeta_{\alpha_\tau}\over \pa\tau}(s)$
at
$\tau=0$
is zero and we can use Corollary \ref{C4.1} to deduce that
$a$
is conformally equivalent to
$\mathbf 1$.
Conversely, if
$a$
is conformally equivalent to
$\mathbf 1$,
then
$\zeta_a=\zeta_{\mathbf 1}=2\zeta_R$.
\hfill $\Box$

\section{Proof of Theorem \ref{Th4.1} and Lemma \ref{L4.1}}  \label{S5}

We again consider a deformation
$\alpha_\tau$
of a positive function
$a\in C^\infty(\S)$
satisfying hypotheses of Lemma \ref{L2.2} with
$l=\infty$.
Hypotheses of all statements in this section coincide with that of Lemma \ref{L2.2}. The hypotheses are not written explicitly for brevity.

\subsection{Proof of Theorem \ref{Th4.1}}
We are going to prove \eqref{4.4} for
$\Re z>2$.
Then the validity of \eqref{4.4} for all
$z\in \C$
will follow by the unique continuation principle since both sides of \eqref{4.4} are entire functions.

The equalities
$\Lambda_aP_0=P_0\Lambda_a=0$
immediately follow from definitions of
these operators (the definition of
$\Lambda_a$
is given in the Introduction and
$P_0$
is defined before Theorem \ref{Th4.1}). We will widely use these equalities with no reference.

Note that
$(\Lambda_a+P_0)^{-z}$
and
$(|D_a|+P_0)^{-z}$
are trace class operators for
$\Re z>2$.
Hence \eqref{3.40} implies that
\begin{equation}
\left.{\pa \big(\zeta_{\alpha_\tau}(z)\big)\over \pa \tau}\right|_{\tau=0}=-z\,\T\big[(\Lambda_a+P_0)^{-z}a^{-1}\beta\big]
+z\,\T\big[(|D_a|+P_0)^{-z}a^{-1}\beta\big].
                                           \label{5.1}
\end{equation}

Recall that the functions
$\phi_{n,\tau}$
are defined in \eqref{2.2}. Setting
$\phi_n=\phi_{n,0}$,
we have the orthonormal basis
$\{\phi_n\}_{n\in\Z}$
consisting of eigenfunctions of the operator
$D_a$
such that
$D_a\phi_n=n\phi_n$. This implies
$(|D_a|+P_0)\phi_n=\max(|n|,1)\phi_n$.

Let us demonstrate that
\begin{equation}
\T\big[(|D_a|+P_0)^{-z}a^{-1}\beta\big]=0.
                                   \label{5.2}
\end{equation}
Indeed, for an arbitrary
$n\in\Z$,
$$
\begin{aligned}
\l(|D_a|+P_0)^{-z}a^{-1}\beta\phi_n,\phi_n\r&=\l a^{-1}\beta\phi_n,(|D_a|+P_0)^{-z}\phi_n\r
=\big(\max(|n|,1)\big)^{-z} \l a^{-1}\beta\phi_n,\phi_n\r\\
&=(2\pi)^{-1}\big(\max(|n|,1)\big)^{-z} \int_{\S} a^{-2}\beta=0.
\end{aligned}
$$
The last equality of the chain is written on the base of \eqref{4.1} since
$\beta=\left.\frac{\pa\alpha_\tau}{\pa\tau}\right|_{\tau=0}$.
From this,
$$
\T\big[(|D_a|+P_0)^{-z}a^{-1}\beta\big]=\sum\limits_{n\in\Z}\l(|D_a|+P_0)^{-z}a^{-1}\beta\phi_n,\phi_n\r=0.
$$
This proves \eqref{5.2}.

In virtue of \eqref{5.2}, formula \eqref{5.1} simplifies to the following one:
\begin{equation}
\left.{\pa \big(\zeta_{\alpha_\tau}(z)\big)\over \pa \tau}\right|_{\tau=0}=-z\,\T\big[(\Lambda_a+P_0)^{-z}a^{-1}\beta\big].
                                   \label{5.3}
\end{equation}

Let the function
$g\in C^\infty(\S)$
be defined by \eqref{4.2}.
Looking at
$a,g$
and
$\beta$
as multiplication operators, we have the equality
$\beta=i(a Dg -gDa)$
which implies
\begin{eqnarray*}
(\Lambda_a+P_0)^{-z}a^{-1}\beta&=&(\Lambda_a+P_0)^{-z}a^{-{1/2}}\beta a^{-1/2}\\
&=&i(\Lambda_a+P_0)^{-z}a^{1/2} D g a^{-{1/2}}-i(\Lambda_a+P_0)^{-z}a^{-{1/2}}gDa^{1/2}.
\end{eqnarray*}
For
$\Re z>2$,
both
$(\Lambda_a+P_0)^{-z}a^{1/2}Dg a^{-{1/2}}$
and
$(\Lambda_a+P_0)^{-z}a^{-{1/2}}gDa^{1/2}$
are trace class operators and we obtain
\begin{equation}
\T\big[(\Lambda_a+P_0)^{-z}a^{-1}\beta\big]=i\,\T\big[(\Lambda_a+P_0)^{-z}a^{1/2}Dg a^{-{1/2}}\big]
-i\,\T\big[(\Lambda_a+P_0)^{-z}a^{-{1/2}}gDa^{1/2}\big].
                                             \label{5.4}
\end{equation}

Recall that the operators
$D,\ \Lambda$
and
$\H$
are related by the equalities
$\H D=D\H=\Lambda$
and
$\H\Lambda=\Lambda\H=D$.
From this,
$$
a^{1/2}D=\Lambda_a a^{-{1/2}}\H, \quad Da^{1/2}=\H a^{-{1/2}}\Lambda_a.
$$
The last of these equalities immediately gives
\begin{equation}
\T\big[(\Lambda_a+P_0)^{-z}a^{-{1/2}}gDa^{1/2}\big]=\T\big[(\Lambda_a+P_0)^{-z}a^{-{1/2}}g\H a^{-{1/2}}\Lambda_a\big].
                                      \label{5.5}
\end{equation}
Using additionally the relation
$(\Lambda_a+P_0)^{-z}\Lambda_a=(\Lambda_a+P_0)^{-z+1}(I-P_0)$,
we easily derive
\begin{equation}
\T\big[(\Lambda_a+P_0)^{-z}a^{1/2}Dg a^{-{1/2}}\big]=\T\big[(\Lambda_a+P_0)^{-z+1}(I-P_0)a^{-{1/2}}\H g a^{-{1/2}}\big].
                                     \label{5.6}
\end{equation}

Rewriting the trace on the right hand side of \eqref{5.5} in terms of a basis of Steklov eigenvectors (eigenvectors of the operator
$\Lambda_a+P_0$)
and again using the relation
 $\Lambda_a(\Lambda_a+P_0)^{-z}=(\Lambda_a+P_0)^{-z+1}(I-P_0)$,
we obtain
\begin{equation}
\begin{aligned}
\T\big[(\Lambda_a+P_0)^{-z}a^{-{1/2}}g\H a^{-{1/2}}\Lambda_a\big]
&=\T\big[\Lambda_a(\Lambda_a+P_0)^{-z}a^{-{1/2}}g\H a^{-{1/2}}\big]\\
&=\T\big[(\Lambda_a+P_0)^{-z+1}(I-P_0)a^{-{1/2}}g\H a^{-{1/2}}\big].
\end{aligned}
                              \label{5.7}
\end{equation}

Collecting \eqref{5.4}, \eqref{5.6} and \eqref{5.7}, we see that
$$
\T\big[(\Lambda_a+P_0)^{-z}a^{-1}\beta\big]=i\,\T\big[(\Lambda_a+P_0)^{-z+1}(I-P_0)a^{-{1/2}}[\H, g] a^{-{1/2}}\big].
$$
Together with \eqref{5.3}, this gives \eqref{4.4}.

We need following easy statement.

\begin{lemma}     \label{L5.1}
For a function
$f\in C^\infty(\S)$,
the operator equalities
\begin{equation}
[\H, \H f]=\H[\H,f]+F_0\,f-{\hat f}_0\,F_0
                               \label{5.8}
\end{equation}
and
\begin{equation}
\Lambda [\H, \H f]\Lambda=\Lambda f\Lambda-DfD
                                    \label{5.9}
\end{equation}
hold, where the operator
$F_0$
maps a function
$u$
to the constant-valued function
${\hat u}_0{\mathbf 1}$.
\end{lemma}

The proof of the lemma is given at the end of this subsection. With the help of the lemma, we now prove \eqref{4.5} for
$g=i\H a$.
Substituting this value into \eqref{4.4}, we obtain
\begin{equation}
\left.{\pa \big(\zeta_{\alpha_\tau}(z)\big)\over \pa \tau}\right|_{\tau=0}
=z\,\T\big[((\Lambda_a+P_0)^{-z+1}(I-P_0)a^{-{1/2}}[\H,\H a]a^{-{1/2}}\big].
                                    \label{5.10}
\end{equation}
Writing the trace on the right hand side of \eqref{5.10} in terms of an orthonormal basis consisting of eigenvectors of the operator
$\Lambda_a+P_0$,
one easily obtains
\begin{equation}
\begin{aligned}
T\big[((\Lambda_a+P_0)^{-z+1}(I\!-\!P_0)a^{-{1/2}}[\H,\H a]a^{-{1/2}}\big]
&=\T\big[((\Lambda_a+P_0)^{-z-1}\Lambda_a a^{-{1/2}}[\H,\H a]a^{-{1/2}}\Lambda_a\big]\\
&=\T\big[((\Lambda_a+P_0)^{-z-1}a^{1/2}\Lambda [\H,\H a]\Lambda a^{1/2}\big].
\end{aligned}
                                    \label{5.11}
\end{equation}
Using \eqref{5.9} with
$f=a$,
we see that
\begin{equation}
\begin{aligned}
\T\big[((\Lambda_a+P_0)^{-z-1}a^{1/2}\Lambda [\H,\H a]\Lambda a^{1/2}\big]
&=\T\big[((\Lambda_a+P_0)^{-z-1}a^{1/2}(\Lambda a\Lambda-DaD) a^{1/2}\big]\\
&=\T\big[((\Lambda_a+P_0)^{-z-1}(\Lambda_a^2-D_a^2)\big].
\end{aligned}
                                    \label{5.12}
\end{equation}
Combining \eqref{5.10}--\eqref{5.12}, we arrive to \eqref{4.5}.
\hfill $\Box$

\begin{proof}[Proof of Lemma \ref{L5.1}]
The alternative definition of the Hilbert transform is as follows: two real functions
$u,v\in C^\infty(\S)$
satisfy
$iv=\H u$ if and only if
$v$
has the zero mean value and
$u+iv$
admits a holomorphic extension to the unit disk.
Thus, for two real functions
$u,v\in C^\infty(\S)$,
the product
$$
(u+\H u)(v+\H v)= uv +(\H u)(\H v)+[u \H v+v\H u]
$$
admits a holomorphic extension to the unit disk. The function in the bracket has the zero mean value as is seen from
$$
\int_S(u \H v+v\H u)=\sum_{n\in \Z\b\{0\}} \hat u_n {\rm sgn}(-n)\hat v_{-n}+\sum_{n\in \Z\b\{0\}} \hat v_n {\rm sgn}(-n)\hat u_{-n}=0.
$$
We have thus proved the product formula
\begin{equation}
\H\big(uv+(\H u)(\H v)\big)=u \H v+v\H u.
                                       \label{5.13}
\end{equation}
Being proved for real smooth functions, the formula is valid for all
$u,v\in L^2(\S)$
since all terms on \eqref{5.13} are bilinear in
$(u,v)$.

Setting
$v=\H f$
in \eqref{5.13}, we have
$$
\H((\H f) u)-(\H f)(\H u)= (\H^2 f) u-\H((\H^2 f)(\H u) ).
$$
Since
$\H^2 f=f-{\hat f}_0{\mathbf 1}$,
the formula becomes
$$
[\H,\H f] u= f u-\H(f(\H u))-{\hat f}_0 u+{\hat f}_0(\H^2 u).
$$
Substituting the expressions
$fu=\H^2(fu)+(\widehat{fu})_0{\mathbf 1}$
and
$\H^2u=u-F_0u$,
we obtain
$$
[\H,\H f] u= \H^2(fu)-\H(f(\H u))+(\widehat{fu})_0{\mathbf 1}-{\hat f}_0(F_0 u).
$$
This can be written in the form
$$
[\H,\H f] u= \H[\H,f]u+(\widehat{fu})_0{\mathbf 1}-{\hat f}_0\,F_0u.
$$
We have thus proved \eqref{5.8}.

Now, we multiply \eqref{5.8} by
$\Lambda$
from both sides and use the obvious equalities
$F_0\Lambda=\Lambda F_0=0$
to obtain
$$
\Lambda[\H, \H f]\Lambda=\Lambda\H[\H,f]\Lambda.
$$
This can be written in the form
$$
\Lambda[\H, \H f]\Lambda=\Lambda\H^2f\Lambda-\Lambda\H f\H\Lambda.
$$
On using the equalities
$\Lambda\H^2=\Lambda$
and
$\Lambda\H=\H\Lambda=D$,
we obtain \eqref{5.9}.
\end{proof}

\subsection{Proof of Lemma \ref{L4.1}}
If $a$ is conformally equivalent to a constant function, then it is of the form
$$
a(\theta)={\hat a}_0+{\hat a}_1e^{i\theta}+{\hat a}_{-1}e^{-i\theta}.
$$
This fact can be easily derived from the definition of conformally equivalent functions and it also follows from \cite[Theorem 1.2]{JS3}. On using this representation, one easily proves that
$$
(\Lambda a\Lambda-DaD)e^{in\theta}=0\quad\mbox{for all}\quad n\in\Z,
$$
i.e.
$\Lambda a\Lambda-DaD=0$.
In the case of a positive function
$a\in C^\infty(\S)$,
this implies
$\Lambda_a^2-D_a^2=0$
and
$\T\big[(\Lambda_a+P_0)^z(\Lambda_a^2-D_a^2)\big]=0$
for any
$z\in\C$.
This proves the ``if'' part of the second statement of Lemma \ref{L4.1}.

\bigskip

Our proof of Lemma \ref{L4.1} is based on some elementary convexity arguments that are actually well known. For the sake of completeness, we present the proof of the following statement.

\begin{lemma} \label{L5.2}
Let
$A:H\rightarrow H$
be a linear operator in a Hilbert space. Assume that there is an orthonormal basis
$\{e_k\}_{k\in{\mathbb N}}$
of
$H$
consisting of eigenvectors of the operator
$A$,
i.e.,
$Ae_k=\lambda_k e_k$
with positive eigenvalues
$\lambda_k$
satisfying
$\lambda_k\le C(k+1)^M$
with some constants
$C$
and
$M$.

Let
$u$
and
$v$
be two vectors from
$H$.
Expand them in the basis
$u=\sum_{k\in{\mathbb N}}\l u,e_k\r\, e_k$,
$v=\sum_{k\in{\mathbb N}}\l v,e_k\r\, e_k$.
Assume that the coefficients of the expansions rapidly decay, i.e.,
$|\l u,e_k\r|+|\l v,e_k\r|\le C_N(k+1)^{-N}$
for every
$N\in{\mathbb N}$.
Assume also that
\begin{equation}
\l u,e_k\r\l e_k,v\r\ge 0\quad \mbox{for every}\ k\in{\mathbb N}
                              \label{5.14}
\end{equation}
and
\begin{equation}
\l u,v\r=\sum_{k\in \N}\l u,e_k\r\l e_k,v\r=1.
                              \label{5.15}
\end{equation}
Then

{\rm (1)}
$\l A^r u,v\r\ge\l Au,v\r^r$
for every
$r\ge1$;

{\rm (2)}
$\l A^ru,v\r\le\l Au,v\r^r$
for every
$r\in[0,1)$.

If, additionally, $A$ is an invertible operator, then

{\rm (3)}
$\l A^ru,v\r\ge\l Au,v\r^r$
for every
$r<0$.
\end{lemma}

\begin{proof}
Let
$f:(0,+\infty)\rightarrow(0,+\infty)$
be a convex function. Then
$$
\l f(A)u,v\r=\sum_{k\in \N}f(\lambda_k)\ep_k,
$$
where
$\ep_k=\l u,e_k\r\l e_k,v\r\ge 0$
and
$\sum_{k\in \N} \ep_k=1$.
We apply the convexity of the function
$f$
to obtain
$$
\l f(A)u,v\r=\sum_{k\in \N}f(\lambda_k)\ep_k\ge f\big(\sum_{k\in \N}\lambda_k\ep_k\big)=f(\l Au,v\r).
$$
In the case of a concave function
$f$,
we have the opposite inequality
$$
\l f(A)u,v\r\le f(\l Au,v\r).
$$

The function
$f_r:(0,+\infty)\rightarrow(0,+\infty),\ f_r(x)=x^r$
is convex for
$r\in(-\infty,0]\cup[1,\infty)$
and is concave for
$r\in[0,1]$.
\end{proof}

We use the orthonormal basis
$\{\phi_n\}_{n\in \Z}$
of
$L^2({\mathbb S})$
which was introduced after formula \eqref{5.1}.
Recall that
$D_a\phi_n=n\phi_n$.
Recall also that
$P_0:L^2({\mathbb S})\rightarrow L^2({\mathbb S})$
is the orthogonal projection onto the one-dimensional subspace spanned by the vector
$\phi_0$.
Then
$\{\phi_n\}_{n\in \Z}$
is the orthonormal basis consisting of eigenfunctions of the invertible operator
$D_a+P_0$
with positive eigenvalues. Recall also that \cite[Lemma 2.1]{JS2}
\begin{equation}
\l \Lambda_a \phi_n,\phi_n\r\ge |n|,
                                   \label{5.16}
\end{equation}
\begin{equation}
\l(\Lambda_a+P_0)^{-1}\phi_n,\phi_n\r\ge |n|^{-1}\quad(n\neq0).
                                   \label{5.17}
\end{equation}

The proof of Lemma \ref{L4.1} consists of 6 parts.

\bigskip

{\bf Part 1.} Let
$s\ge 2$.
Fix an integer
$n\not=0$.
Set
$A=(\Lambda_a+P_0)^{s-1}:L^2({\mathbb S})\rightarrow L^2({\mathbb S})$,
$u=v=\phi_n$
and
$r=\frac{s+1}{s-1}\ge1$
in Lemma \ref{L5.2}.
Hypotheses of the Lemma are satisfied since
$(\Lambda_a+P_0)^{s-1}$
is a positive self-adjoint operator and
$\phi_n$
is a unit vector in
$L^2({\mathbb S})$.
Applying statement (1) of Lemma \ref{L5.2}, we obtain
\begin{equation}
\l(\Lambda_a\!+\!P_0)^{s+1}\phi_n,\phi_n\r\ge \l(\Lambda_a\!+\!P_0)^{s-1}\phi_n,\phi_n\r^{s+1\over s-1}=\l(\Lambda_a\!+\!P_0)^{s-1}\phi_n,\phi_n\r\l(\Lambda_a\!+\!P_0)^{s-1}\phi_n,\phi_n\r^{2\over s-1}.
                               \label{5.18}
\end{equation}
Then we set
$A=\Lambda_a+P_0$,
$u=v=\phi_n$
and
$r=s-1\ge1$
in Lemma \ref{L5.2}. By the same statement (1) of Lemma \ref{L5.2},
$$
\l(\Lambda_a+P_0)^{s-1}\phi_n,\phi_n\r\ge \l\Lambda_a\phi_n,\phi_n\r^{s-1}.
$$
With the help of \eqref{5.16}, this gives
\begin{equation}
\l(\Lambda_a+P_0)^{s-1}\phi_n,\phi_n\r\ge |n|^{s-1}.
                                 \label{5.19}
\end{equation}
We combine \eqref{5.18} and \eqref{5.19} to obtain
$$
\l(\Lambda_a+P_0)^{s+1}\phi_n,\phi_n\r\ge \l(\Lambda_a+P_0)^{s-1}\phi_n,\phi_n\r n^2=\l(\Lambda_a+P_0)^{s-1}D_a^2\phi_n,\phi_n\r.
$$
This inequality holds for every
$n\in{\mathbb Z}\setminus\{0\}$.
It implies
\begin{equation}
\T((\Lambda_a+P_0)^{s-1}(\Lambda_a^2-D_a^2))=\sum_{n\not=0}\Big[\l(\Lambda_a+P_0)^{s+1}\phi_n,\phi_n\r-\l(\Lambda_a+P_0)^{s-1}D_a^2\phi_n,\phi_n\r\Big]\ge 0.
                                 \label{5.20}
\end{equation}
We have thus proved \eqref{4.6} for
$s\ge2$.

The equality in \eqref{5.20} holds if and only if each summand on the right-hand side is zero. In such a case, the equality in \eqref{5.16} must hold for  every
$n\in{\mathbb Z}\setminus\{0\}$.
In particular, setting
$n=1$
in \eqref{5.16}, we have
$\l \Lambda_a \phi_1,\phi_1\r=1$.
We can now use \cite[Lemma 2.5]{JS2} to obtain that
$a$
is conformally equivalent to the constant function
${\mathbf 1}$.

\bigskip

{\bf Part 2.}
Let
$s\le -2$.
Here our arguments repeat that of Part 1 with one exception:
$(\Lambda_a+P_0)^{-1}$
now plays the role of
$\Lambda_a+P_0$.
By Lemma \ref{L5.2} and \eqref{5.17},
$$
\begin{aligned}
\l(\Lambda_a+P_0)^{s-1}\phi_n,\phi_n\r&\ge \l(\Lambda_a+P_0)^{s+1}\phi_n,\phi_n\r^{s-1\over s+1}\\
&\ge\l(\Lambda_a+P_0)^{s+1}\phi_n,\phi_n\r^{-{2\over s+1}}\l(\Lambda_a+P_0)^{s+1}\phi_n,\phi_n\r\\
&\ge n^{-2}\l(\Lambda_a+P_0)^{s+1}\phi_n,\phi_n\r,
\end{aligned}
$$
i.e.
$$
\l(\Lambda_a+P_0)^{s+1}\phi_n,\phi_n\r\le n^2\l(\Lambda_a+P_0)^{s-1}\phi_n,\phi_n\r=\l(\Lambda_a+P_0)^{s-1}D_a^2\phi_n,\phi_n\r.
$$
We have used that
$s+1\le -1$,
$-2/(s+1)>0$, $(s-1)/(s+1)\ge 1$.
We conclude as in Part 1.

\bigskip

{\bf Part 3.}
Let
$1\le s\le 2$.
Fix an integer
$n\not=0$.
We first set
$A=\Lambda_a+P_0$,
$u=v=\phi_n$
and
$r=s-1\in(0,1)$
in Lemma \ref{L5.2}. By statement (2) of Lemma \ref{L5.2},
\begin{equation}
\l(\Lambda_a+P_0)^{s-1}\phi_n,\phi_n\r\le \l\Lambda_a\phi_n,\phi_n\r^{s-1}.
                                      \label{5.21}
\end{equation}
Then we set
$A=\Lambda_a+P_0$,
$u=v=\phi_n$
and
$r=s+1\ge2$
in Lemma \ref{L5.2}. Applying statement (1) of Lemma \ref{L5.2}, we obtain
$$
\l(\Lambda_a+P_0)^{s+1}\phi_n,\phi_n\r\ge\l\Lambda_a\phi_n,\phi_n\r^{s+1}=\l\Lambda_a\phi_n,\phi_n\r^{s-1}\l\Lambda_a\phi_n,\phi_n\r^2.
$$
With the help of \eqref{5.16}, this gives
\begin{equation}
\l(\Lambda_a+P_0)^{s+1}\phi_n,\phi_n\r\ge n^2\l\Lambda_a\phi_n,\phi_n\r^{s-1}.
                             \label{5.22}
\end{equation}
Inequalities \eqref{5.21} and \eqref{5.22} imply
\begin{equation}
\T((\Lambda_a+P_0)^{s-1}(\Lambda_a^2-D_a^2))=\sum_{n\in \Z\b\{0\}}\big(\l(\Lambda_a+P_0)^{s+1}\phi_n,\phi_n\r-n^2\l\Lambda_a^{s-1}\phi_n,\phi_n\r\big)\ge 0.
                             \label{5.23}
\end{equation}
We have thus proved \eqref{4.6} for $s\in(1,2)$.

Again the equality in \eqref{5.23} means that each summand on the right-hand side is zero. In such a case, the equality in \eqref{5.16} must hold for  every $n\in{\mathbb Z}\setminus\{0\}$.
In particular, setting
$n=1$
in \eqref{5.16}, we have
$\l \Lambda_a \phi_1,\phi_1\r=1$.
We use again \cite[Lemma 2.5]{JS2} to obtain that
$a$
is conformally equivalent to
${\mathbf 1}$.

\bigskip

{\bf Part 4.}
Let
$-2\le s\le -1$.
Here our arguments repeat that of Part 3 with the exception:
$(\Lambda_a+P_0)^{-1}$
now plays the role of
$\Lambda_a+P_0$.
With the help of Lemma \ref{L5.2} and of \eqref{5.17}, we obtain the estimates
$$
\l(\Lambda_a+P_0)^{s+1}\phi_n,\phi_n\r\le\l(\Lambda_a+P_0)^{-1}\phi_n,\phi_n\r^{-s-1}
$$
and
$$
\begin{aligned}
\l(\Lambda_a+P_0)^{s-1}\phi_n,\phi_n\r&\ge \l(\Lambda_a+P_0)^{-1}\phi_n,\phi_n\r^{1-s}\\
&=\l(\Lambda_a+P_0)^{-1}\phi_n,\phi_n\r^{-1-s}\l(\Lambda_a+P_0)^{-1}\phi_n,\phi_n\r^2\\
&\ge n^{-2}\l(\Lambda_a+P_0)^{-1}\phi_n,\phi_n\r^{-1-s}.
\end{aligned}
$$
Hence
$$
\T((\Lambda_a+P_0)^{s-1}(\Lambda_a^2-D_a^2))=\sum_{n\in \Z}\big(\l(\Lambda_a+P_0)^{s+1}\phi_n,\phi_n\r-n^2\l(\Lambda_a+P_0)^{s-1}\phi_n,\phi_n\r\big)\le 0. $$
This proves \eqref{4.7} for
$s\in(-2,-1)$.

Again the equality implies
$\l(\Lambda_a+P_0)^{-1}\phi_1,\phi_1\r=1$
and we use \cite[Lemma 2.5]{JS2} to obtain that
$a$
is conformally equivalent to
${\mathbf 1}$.

\bigskip

{\bf Part 5.}
Let
$0<s\le 1$.
First of all, on using the equality
$D_a=\Lambda_a a^{-{1/2}}\H a^{1/2}$
we write
\begin{equation}
\begin{aligned}
&\T((\Lambda_a+P_0)^{s-1}(\Lambda_a^2-D_a^2))=\\
&=\sum_{n\in\Z\b\{0\}}\Big(\l(\Lambda_a+P_0)^{s+1}\phi_n,\phi_n\r-\l(\Lambda_a+P_0)^{s-1}\Lambda_a a^{-1/2}\H a^{1/2}D_a\phi_n,\phi_n\r\Big)\\
&=\sum_{n\in\Z\b\{0\}}\Big(\l(\Lambda_a+P_0)^{s+1}\phi_n,\phi_n\r-n\l(\Lambda_a+P_0)^s a^{-1/2}\H a^{1/2}\phi_n,\phi_n\r\Big).
\end{aligned}
                               \label{5.24}
\end{equation}
The first term of each summand on the right-hand side of \eqref{5.24} is real since
$(\Lambda_a+P_0)^{s+1}$
is a self-adjoint operator. We will see that the second term is also real, although it is not quite obvious.

By statement (1) of Lemma \ref{L5.2} and by \eqref{5.16},
\begin{equation}
\l(\Lambda_a+P_0)^{s+1}\phi_n,\phi_n\r\ge \big(\l\Lambda_a\phi_n,\phi_n\r\big)^{s+1}\ge |n|^s\l\Lambda_a \phi_n,\phi_n\r.
                                           \label{5.25}
\end{equation}

Let
$\{\Psi_k\}_{k\in \N}$
be the  orthonormal basis of
$L^2(\S)$
consisting of eigenfunctions of the operator
$A=\Lambda_a+P_0$
(it is the partial case, for
$\tau=0$,
of the basis
$\{\Psi_{k,\tau}\}_{k\in \N}$
that was used in Section 2.3).
Then
$(\Lambda_a+P_0)\Psi_k=\tilde\lambda_k \Psi_k$
for
$k\in \N$,
where
$\tilde \lambda_0=1$,
$\Psi_0=\phi_0$,
and
$\tilde\lambda_k=\lambda_k$
for
$k>0$
($\lambda_k$
being the Steklov eigenvalues of the function
$a$).

Let us fix an integer
$n\neq0$.
To estimate the second term of the summand on the right-hand side of \eqref{5.24}, we use Lemma \ref{L5.2} with
\begin{equation}
e_k=\Psi_k,\quad A=\Lambda_a+P_0,\quad u=\delta_n^{-1}a^{-1/2}{\mathcal H}a^{1/2}\phi_n,\quad v=\mbox{sgn}(n)\phi_n,\quad r=s\in(0,1),
                                           \label{5.26}
\end{equation}
where the positive constant
$\delta_n$
will be chosen later.

We have to check the hypotheses \eqref{5.14}--\eqref{5.15} of Lemma \ref{L5.2}. To this end we write
$$
\l u,\Psi_k\r\l \Psi_k,v\r=\delta_n^{-1}\mbox{sgn}(n)\l a^{-1/2}{\mathcal H}a^{1/2}\phi_n,\Psi_k\r\l \Psi_k,\phi_n\r.
$$
In the case of
$k=0$,
the right-hand side is equal to zero since
$\l e_0,\phi_n\r=0$.
In the case of
$k>0$, we use the equalities
$\Lambda_a\Psi_k=\lambda_k\Psi_k$
and
$D_a\phi_n=n\phi_n$ to obtain
$$
\begin{aligned}
\l u,\Psi_k\r\l \Psi_k,v\r&=\delta_n^{-1}\mbox{sgn}(n)\lambda_k^{-1}\l a^{-1/2}{\mathcal H}a^{1/2}\phi_n,\Lambda_a \Psi_k\r\l \Psi_k,\phi_n\r\\
&=\delta_n^{-1}\mbox{sgn}(n)\lambda_k^{-1}\l \Lambda_a\, a^{-1/2}{\mathcal H}a^{1/2}\phi_n,\Psi_k\r\l \Psi_k,\phi_n \r\\
&=\delta_n^{-1}\mbox{sgn}(n)\lambda_k^{-1}\l D_a\phi_n,\Psi_k\r\l \Psi_k,\phi_n\r\\
&=\delta_n^{-1}\mbox{sgn}(n)n\lambda_k^{-1}\l \phi_n,\Psi_k\r\l \Psi_k,\phi_n\r\\
&=\delta_n^{-1}|n|\lambda_k^{-1}|\l \phi_n,\Psi_k\r|^2\ge0.
\end{aligned}
$$
This proves \eqref{5.14}. Equality \eqref{5.15} looks now as follows:
$$
\l u,v\r=\sum_{k\in \N}\l u,\Psi_k\r\l \Psi_k,v\r=\delta_n^{-1}|n|\sum_{k>0}\lambda_k^{-1}|\l \phi_n,\Psi_k\r|^2=1.
$$
To satisfy this hypothesis, we set
$$
\delta_n=|n|\sum_{k>0}\lambda_k^{-1}|\l \phi_n,\Psi_k\r|^2.
$$
Observe also that
\begin{equation}
\delta_n\ge1.
                                           \label{5.27}
\end{equation}
Indeed, as we have seen
$$
\begin{aligned}
|n|\lambda_k^{-1}|\l \phi_n,\Psi_k\r|^2&=\mbox{sgn}(n)\lambda_k^{-1}\l  a^{-1/2}{\mathcal H}a^{1/2}\phi_n,\Lambda_a\,\Psi_k\r\l \Psi_k,\phi_n \r\\
&=\mbox{sgn}(n)\l  a^{-1/2}{\mathcal H}a^{1/2}\phi_n,\Psi_k\r\l \Psi_k,\phi_n \r.
\end{aligned}
$$
Therefore
$$
\begin{aligned}
\delta_n=\sum_{k>0}|n|\lambda_k^{-1}|\l \phi_n,\Psi_k\r|^2
&=\mbox{sgn}(n)\sum_{k>0}\l  a^{-1/2}{\mathcal H}a^{1/2}\phi_n,\Psi_k\r\l \Psi_k,\phi_n \r\\
&=\mbox{sgn}(n)\Big\langle a^{-1/2}{\mathcal H}a^{1/2}\phi_n,\sum_{k>0}\l\phi_n,\Psi_k\r \Psi_k\Big\rangle\\
&=\mbox{sgn}(n)\l a^{-1/2}{\mathcal H}a^{1/2}\phi_n,\phi_n\r\\
&=|n|^{-1}\l a^{-1/2}{\mathcal H}a^{1/2}\phi_n,D_a\phi_n\r\\
&=|n|^{-1}\l D_aa^{-1/2}{\mathcal H}a^{1/2}\phi_n,\phi_n\r.
\end{aligned}
$$
Since $D_aa^{-1/2}{\mathcal H}a^{1/2}=\Lambda_a$, we obtain
\begin{equation}
\delta_n=|n|^{-1}\l \Lambda_a\phi_n,\phi_n\r.
                                           \label{5.28}
\end{equation}
With the help of \eqref{5.16}, this gives \eqref{5.27}.

Thus, hypotheses of of Lemma \ref{L5.2} are satisfied. Applying statement (2) of Lemma \ref{L5.2}, we obtain
\begin{equation}
\mbox{sgn}(n)\delta_n^{-1}\l(\Lambda_a+P_0)^s a^{-1/2}{\mathcal H}a^{1/2}\phi_n,\phi_n\r=\l A^su,v\r\le\l Au,v\r^s.
                                           \label{5.29}
\end{equation}

Next, we compute on the base of \eqref{5.26}
$$
\begin{aligned}
\l Au,v\r&=\delta_n^{-1}\l(\Lambda_a+P_0)a^{-1/2}{\mathcal H}a^{1/2}\phi_n,\mbox{sgn}(n)\phi_n\r\\
&=\delta_n^{-1}\l a^{-1/2}{\mathcal H}a^{1/2}\phi_n,\mbox{sgn}(n)(\Lambda_a+P_0)\phi_n\r\\
&=\delta_n^{-1}\l a^{-1/2}{\mathcal H}a^{1/2}\phi_n,\mbox{sgn}(n)\Lambda_a\phi_n\r\\
&=\delta_n^{-1}\l \Lambda_a a^{-1/2}{\mathcal H}a^{1/2}\phi_n,\mbox{sgn}(n)\phi_n\r,
\end{aligned}
$$
i.e.,
\begin{equation}
\l Au,v\r=\delta_n^{-1}\l \Lambda_a a^{-1/2}{\mathcal H}a^{1/2}\phi_n,\mbox{sgn}(n)\phi_n\r.
                                           \label{5.30}
\end{equation}
Since
$$
\Lambda_aa^{-1/2}{\mathcal H}a^{1/2}\phi_n=D_a\phi_n=n\phi_n,
$$
equality \eqref{5.30} simplifies to the following one:
\begin{equation}
\l Au,v\r=
\delta_n^{-1}|n|\|\phi_n\|^2=\delta_n^{-1}|n|.
                                           \label{5.31}
\end{equation}

From \eqref{5.29} and \eqref{5.26}, we obtain
$$
\mbox{sgn}(n)\l(\Lambda_a+P_0)^s a^{-1/2}{\mathcal H}a^{1/2}\phi_n,\phi_n\r\le\delta_n^{1-s}|n|^s.
$$
Multiplying this inequality by
$|n|$,
we have
$$
n\l(\Lambda_a+P_0)^s a^{-1/2}{\mathcal H}a^{1/2}\phi_n,\phi_n\r\le\delta_n^{1-s}|n|^{s+1}=\Big(\frac{|n|}{\delta_n}\Big)^s|n|\delta_n.
$$
With the help of \eqref{5.27} and \eqref{5.28}, this gives
\begin{equation}
n\l(\Lambda_a+P_0)^s a^{-1/2}{\mathcal H}a^{1/2}\phi_n,\phi_n\r\le|n|^s\l\Lambda_a\phi_n,\phi_n\r.
                                           \label{5.32}
\end{equation}

Inequality \eqref{5.32} holds for every
$n\in{\mathbb Z}\setminus\{0\}$.
Together with \eqref{5.25}, it means that all summands on the right-hand side of \eqref{5.24} are non-negative. This proves \eqref{4.6} for
$s\in(0,1)$.

Equality in \eqref{4.6} implies that each summand in \eqref{5.24} is zero, which means equality in \eqref{5.25}. For
$n=1$
it implies
$\l\Lambda_a\phi_1,\phi_1\r=1$
and we conclude as before.

\bigskip

{\bf Part 6.} Let
$s\in (-1,0)$.
We repeat our arguments of Part 5. Formula \eqref{5.24} is still valid.  But instead of \eqref{5.25} we have now the opposite inequality
\begin{equation}
\l(\Lambda_a+P_0)^{s+1}\phi_n,\phi_n\r\le \l\Lambda_a \phi_n,\phi_n\r^{s+1}.
                               \label{5.33}
\end{equation}
Indeed, since
$0\le s+1\le 1$,
we have to apply statement (2) of Lemma \ref{L5.2}.

All our formulas in between \eqref{5.24} and \eqref{5.29} remain valid. But instead of \eqref{5.29} we have now the opposite inequality
\begin{equation}
\mbox{sgn}(n)\delta_n^{-1}\l(\Lambda_a+P_0)^s a^{-1/2}{\mathcal H}a^{1/2}\phi_n,\phi_n\r=\l A^su,v\r\ge\l Au,v\r^s.
                                           \label{5.34}
\end{equation}
Indeed, we have to apply statement (3) of Lemma \ref{L5.2}.

Formula \eqref{5.31} is still valid. From \eqref{5.31} and \eqref{5.34}, we obtain
$$
\mbox{sgn}(n)\l(\Lambda_a+P_0)^s a^{-1/2}{\mathcal H}a^{1/2}\phi_n,\phi_n\r\ge\delta_n^{1-s}|n|^s.
$$
Multiplying this inequality by
$|n|$,
we have
$$
n\l(\Lambda_a+P_0)^s a^{-1/2}{\mathcal H}a^{1/2}\phi_n,\phi_n\r\ge\delta_n^{1-s}|n|^{s+1}=\Big(\frac{\delta_n}{|n|}\Big)^{-s}|n|\delta_n.
$$
Substituting the value
$|n|\delta_n=\l\Lambda_a\phi_n,\phi_n\r$
from \eqref{5.28}, we arrive to the inequality
$$
n\l(\Lambda_a+P_0)^s a^{-1/2}{\mathcal H}a^{1/2}\phi_n,\phi_n\r\ge|n|^{s}\delta_n^{-s}\l\Lambda_a\phi_n,\phi_n\r.
$$
Then, substituting the value
$\delta_n^{-s}=|n|^{s}\l\Lambda_a\phi_n,\phi_n\r^{-s}$
from \eqref{5.28}, we obtain
$$
n\l(\Lambda_a+P_0)^s a^{-1/2}{\mathcal H}a^{1/2}\phi_n,\phi_n\r\ge|n|^{2s}\l\Lambda_a\phi_n,\phi_n\r^{1-s}.
$$
We rewrite this in the form
$$
n\l(\Lambda_a+P_0)^s a^{-1/2}{\mathcal H}a^{1/2}\phi_n,\phi_n\r\ge|n|^{2s}\l\Lambda_a\phi_n,\phi_n\r^{s+1}\l\Lambda_a\phi_n,\phi_n\r^{-2s}
$$
and use the inequality
$\l\Lambda_a\phi_n,\phi_n\r^{-2s}\ge |n|^{-2s}$
that follows from \eqref{5.16} (recall that
$-2s>0$) to obtain
\begin{equation}
n\l(\Lambda_a+P_0)^s a^{-1/2}{\mathcal H}a^{1/2}\phi_n,\phi_n\r\ge\l\Lambda_a\phi_n,\phi_n\r^{s+1}.
                                           \label{5.35}
\end{equation}

Inequality \eqref{5.35} holds for every
$n\in{\mathbb Z}\setminus\{0\}$.
Together with \eqref{5.33}, it means that all summands on the right-hand side of \eqref{5.24} are non-positive. This proves \eqref{4.7} for 
$s\in(-1,0)$.

Equality in \eqref{4.7} implies that each summand in \eqref{5.24} is zero, which means equality in \eqref{5.25}. For
$n=1$
it implies
$\l\Lambda_a\phi_1,\phi_1\r=1$
and we conclude as before.

\section{Proof of Theorem \ref{Th1.3}}

\subsection{A compactness lemma}
Our proof of Theorem  \ref{Th1.3} heavily relies on invariance of compact sets in
$C^\infty(\S)$
under the flow of the equation \eqref{1.8}. The compact sets can be determined in terms of the Steklov zeta function and the determination takes its roots from \cite{JS3}. We have the following result.

\begin{theorem} \label{Th6.1}
Let
$c=\{c_k\}_{k\in \N}$
be a sequence of positive reals. The subset
$\K_c$ of $C^\infty(\S)$,
defined by
\begin{equation}
\K_c=\Big\{0<b\in C^\infty(\S)\mid \int_{\S}\!\!b^{-1}\!=\!2\pi, \hat b_0\le c_0,\ \zeta_b(-1)\le c_1, \zeta_b(-2m)\le c_{m+1},(m\!=\!1,2,\dots)\Big\},
                                         \label{6.1}
\end{equation}
is compact in
$C^\infty(\S)$.
In particular, there exists
$\ep_c>0$
dependent on
$c_0$
and
$c_1$
such that
\begin{equation}
\ep_c\le b \le\ep_c^{-1}\quad\textrm{for any}\quad b\in\K_c.
                                       \label{6.2}
\end{equation}
Additionally, for any positive integer
$m$,
\begin{equation}
\sup_{b\in\K_c}\| b\|_{C^m(\S)}\le C_m,\quad \sup_{b\in\K_c}\|\H b\|_{C^m(\S)}\le C_m
                                      \label{6.3}
\end{equation}
with a constant
$C_m$
that depends on
$c_0,\ldots, c_{m+2}$
only.
\end{theorem}

The values
$Z_m(b)=\zeta_b(-2m)\ (m=1,2,\dots)$
are the so-called {\it zeta invariants} of the function
$b$
introduced in \cite{MS}.

\begin{proof}
The proof mostly follows that of \cite[Lemma 5.3]{JS3} on the compactness of a Steklov isospectral family of planar domains.
We will stress only the differences between the latter proof and the current proof of Theorem \ref{Th6.1}.

The main difference between the two proofs is the first step where one needs to control the zeroth Fourier coefficient $\hat b_0$ and the uniform norm $\|\ln(b)\|_{\infty}$.
This was done by Edward \cite{E2} and repeated in \cite[Lemmas 5.1 and 5.2]{JS3}.

Here we provide details of the first step. The control of the zeroth Fourier coefficient is granted by the definition of
$\K_c$:
\begin{equation}
0\le \hat b_0\le c_0\quad\textrm{for}\quad b\in \K_c.
                           \label{6.4}
\end{equation}

Now  we recall Kogan's formula \cite{Ko}:
\begin{equation}
\zeta_b(-1)=\frac{1}{12\pi}\int\limits_0^{2\pi}\Big(\frac{\big(b'(\theta)\big)^2}{b(\theta)}-b(\theta)\Big)\,d\theta=-\frac{1}{6}\hat b_0+\frac{1}{12\pi}\int\limits_0^{2\pi}\frac{\big(b'(\theta)\big)^2}{b(\theta)}\,d\theta,
                                   \label{6.5}
\end{equation}
for a smooth positive function
$b$
on
$\S$.

Now let
$b\in \K_c$.
Combining \eqref{6.4}, \eqref{6.5} and the definition \eqref{6.1} of
$\K_c$,
we obtain
\begin{equation}
\int\limits_0^{2\pi}\frac{\big(b'(\theta)\big)^2}{b(\theta)}\,d\theta=12\pi(\zeta_b(-1)+\hat b_0/6)\le 2\pi(6c_1+c_0).
                      \label{6.6}
\end{equation}
Then by Bunyakovsky-Cauchy-Schwarz inequality we have
$$
\int_S|\ln(b)'|=\int_{\S}{|b'|\over b^{1/2}}b^{-1/2}\le \left(\int_S{(b')^2\over b}\right)^{1/2}\left(\int_{\S}b^{-1}\right)^{1/2}
\le 2\pi(6c_1+c_0)^{1/2}.
$$
We have used \eqref{6.6} and the normalization condition
$\int_{\S}b^{-1}=2\pi$
satisfied by any
$b\in \K_c$.

Then we prove a uniform control of the
$L^1$-norm of
$\ln(b)$
with respect to the constants
$c_0$
and
$c_1$.
As in \cite{E2} we can conclude that
$$
\|\ln(b)\|_\infty\le 2\pi(6c_1+c_0)^{1/2}.
$$
Indeed, the normalization condition also tells us that there exists
$\theta_0\in [0,2\pi)$
such that
$b(\theta_0)=1$.
We can assume without lost of generality that
$\theta_0=0$,
and we have
$$
\ln\big(b(\theta)\big)=\int_0^\theta  \ln(b)'(s)\,ds,\quad |\ln(b)(\theta)|\le 2\pi(6c_1+c_0)^{1/2}\quad \big(\theta\in [0,2\pi)\big).
$$
The first step is completed. Note also that the bound on
$\|\ln(b)\|_\infty$
provides the right value for
$\ep_c$. Here
$\ep_c= \exp\big(-2\pi(6c_1+c_0)^{1/2}\big)$
would fit in the second statement of the theorem.

The second step is a repetition of the proof of \cite[Lemma 5.3]{JS3}.
In the latter proof, zeta invariants
$Z_m(b)=\zeta_b(-2m)$
have fixed values for
$b$
belonging to a specific subset of
$C^\infty(\S)$.
Now we use that  the zeta invariants of a function
$b\in \K_c$
are bounded:
$Z_m(b)=\zeta_b(-2m)\le c_{m+1}$.
This is enough to conclude that
$$
\sup_{b\in\K_c}\|b\|_{H^m(\S)}\le C_m
$$
for any
$m\in \N$
with some constant
$C_m$
depending on
$c_0,\ldots, c_{m+1}$.
We also observe that
$\|\H b\|_{H^m(\S)}\le \|b\|_{H^m(\S)}$
for any positive integer
$m$
and any
$b\in C^\infty(\S)$.
Then we use the embedding of
$H^{m+1}(\S)$
into
$C^m(\S)$ to obtain \eqref{6.3}.
\end{proof}

\subsection{Basic properties of the flow \eqref{1.8}}

We will use the following basic statement for the quadratic form on the right-hand side of \eqref{1.8}.

\begin{lemma} \label{L6.1}
For a real function
$b\in C^\infty(\S)$,
define
$$
\B(b)=-b (\Lambda b)+(\H b) (Db).
$$
Then
\begin{equation}
\B(b)=-4\Re(b_+(\Lambda \bar b_+)),
                                 \label{6.7}
\end{equation}
where
$$
b_+(\theta)={\hat b_0\over 2}+\sum_{k\ge 1}\hat b_ke^{ik\theta}.
$$

If
$\hat b_k=0$
for
$|k|> N$
with some
$N\in\N$,
then also
\begin{equation}
(\widehat{\B(b)})_k=0\quad\textrm{for}\quad|k|>N
                               \label{6.8}
\end{equation}
and
\begin{equation}
(\widehat{\B(b)})_k= -k\, \hat b_0\, \hat b_k + \sum_{1\le l,m\le N,\ l-m=k} (l+m)\,\hat b_l\,\overline{\hat b_{m}}\qquad (0\le k\le N).
                                 \label{6.9}
\end{equation}
(By convention, a sum over an empty set is zero.)
\end{lemma}

\begin{proof}
Since $b$ is a real function,
$b=b_++\bar b_+,\ \H b= b_+-\bar b_+$
and
\begin{eqnarray*}
-b (\Lambda b)+(\H b) (Db)&=& -b (\Lambda b)+(\H b) \Lambda \H b\\
&=&-(b_+ +\bar b_+)\Lambda(b_+ +\bar b_+)+(b_+ -\bar b_+)\Lambda (b_+ -\bar b_+)\\
&=&-2b_+\Lambda \bar b_+-2\bar b_+\Lambda b_+=-4\Re(b_+(\Lambda \bar b_+)).
\end{eqnarray*}
This proves \eqref{6.7}.

Assume now that
$\hat b_k=0$
for
$|k|> N$.
This means that
$$
b_+(\theta)={\hat b_0\over 2}+\sum_{k= 1}^N\hat b_ke^{ik\theta},\quad \bar b_+(\theta)={\hat b_0\over 2}+\sum_{k= 1}^N\hat b_{-k}e^{-ik\theta}
$$
and we have
$$
\B(b)=-4\Re \Big[{\hat b_0\over 2}\,\sum_{k= 1}^N k\,\hat b_{-k}\,e^{-ik\theta}
+ \sum_{1\le k,l\le N} k\,\hat b_{-k}\hat b_l\,e^{i(l-k)\theta}\Big].
$$
Now \eqref{6.8} and \eqref{6.9} are obvious.
\end{proof}

We will also use the following property of the flow \eqref{1.8}.

\begin{lemma}  \label{L6.2}
Let
$I$
be a real interval and let
$\alpha\in C^\infty(I,C^\infty(\S))$
be a real solution to \eqref{1.8}. Then the mean value
$\int_{\S}\alpha_\tau$
is a non-increasing function of
$\tau\in I$ and
\begin{equation}
\frac{\pa}{\pa\tau} \int_\S \alpha_\tau =-4\l \alpha_{\tau,+},\Lambda \alpha_{\tau,+} \r\quad (\tau\in I).
                                         \label{6.10}
\end{equation}
Additionally, if
$\frac{\pa}{\pa\tau} \int_\S \alpha_\tau=0$
for some
$\tau\in I$,
then
$\alpha_{\tau}=(\widehat{\alpha_\tau})_0=\textrm{\rm const}$.
\end{lemma}

\begin{proof}
We average \eqref{1.8} and use \eqref{6.7} to obtain
$$
\frac{\pa}{\pa\tau} \int_\S \alpha_\tau=-4 \Re\int_{\S}\big(\alpha_{\tau,+} \Lambda \overline \alpha_{\tau,+}\big)=-4\l \alpha_{\tau,+},\Lambda \alpha_{\tau,+} \r\le 0.
$$
The equality here holds if and only if $\alpha_{\tau,+}$ is a constant function.
\end{proof}

The normalization condition \eqref{1.4} is preserved by the flow.

\begin{lemma} \label{L6.3}
Let
$\alpha\in C^\infty(I,C^\infty(\S))$
be a solution to \eqref{1.8} on a real interval such that
$\alpha_\tau$
is a positive function for any
$\tau\in I$.
Then
$\int_{\S}\alpha_\tau^{-1}$
is independent of
$\tau$.
\end{lemma}

\begin{proof}
We derive from \eqref{1.8}
$$
 {\pa\alpha_\tau^{-1}\over\pa\tau}=- \alpha_\tau^{-2}{\pa\alpha_\tau\over\pa\tau}
 = \alpha_\tau^{-2}(\alpha_\tau \Lambda\alpha_\tau - \H\alpha_\tau D\alpha_\tau)
 = -\alpha_\tau^{-2}(\alpha_\tau g_\tau'-g_\tau \alpha_\tau')=-\Big({g_\tau\over \alpha_\tau}\Big)',
$$
where
$g_\tau=i\H\alpha_\tau$.
Averaging over
$\S$,
we obtain
$$
\frac{\pa}{\pa\tau} \int_{\S}\alpha_\tau^{-1}=\int_{\S}\frac{\pa}{\pa\tau}\alpha_\tau^{-1}=-\int_{\S}\Big({g_\tau\over \alpha_\tau}\Big)'=0.
$$
\end{proof}

\subsection{Reduction to a system of ODE's}
We prove here a weaker version of Theorem \ref{Th1.3} such that the initial data for equation \eqref{1.8} have a finite amount of nonzero Fourier modes.

\begin{theorem} \label{Th6.2}
Assume a positive function
$a\in C^\infty(\S)$
to satisfy the normalization condition \eqref{1.4} and to be such that
$\hat a_k=0$
for
$|k|> N$
with some
$N\in \N$.
Then there exists a unique smooth path
$\alpha\in C^\infty([0,\infty),C^\infty(\S))$
of positive functions such that
\begin{eqnarray}
&&\alpha_0=a,
                              \label{6.11}\\
&&{\pa\alpha_\tau\over \pa\tau}=-\alpha_\tau (\Lambda \alpha_\tau)+(\H\alpha_\tau) (D\alpha_\tau)\quad\textrm{\rm for}\quad \tau\in[0,\infty),
                                 \label{6.12}\\
&&(\widehat{\alpha_\tau})_k=0\quad   \textrm{\rm for}\quad \tau\in[0,\infty)\quad \textrm{\rm and}\quad |k|> N.
                              \label{6.13}
\end{eqnarray}

Additionally, if
$\K_c$
is a compact set in
$C^\infty(\S)$
defined by \eqref{6.1} for some sequence
$\{c_k\}_{k\in \N}$
of positive reals such that
$$
c_0\ge \hat a_0,\quad c_1\ge \zeta_a(-1),\quad c_{m+1}\ge \zeta_a(-2m)\ (m\in \N\b\{0\}),
$$
then
$\alpha_\tau \in \K_c$
for any
$\tau\in [0,\infty)$.
\end{theorem}

\begin{proof}
Of course
$a\in\K_c$
and we denote by
$\ep_c$
the constant that appears in \eqref{6.2}.

Now, let
$v\in \K_c$
be such that
$\hat v_k=0$
for
$|k|> N$.
We consider the differential equation \eqref{6.12} with the initial data
$v$.
Due to \eqref{6.9}, we translate \eqref{6.12} into the system of ODE's for Fourier coefficients of the smooth path
$\alpha$:
\begin{equation}
{\pa \alpha_{k,\tau}\over \pa\tau}= -k\,   \alpha_{0,\tau} \alpha_{k,\tau} + \sum_{1\le l,m\le N,\ l-m=k}(l+m)\,\alpha_{l,\tau}\,\overline{\alpha_{m,\tau}}
\quad (0\le k\le N, \tau\in \R)
                                         \label{6.14}
\end{equation}
with the initial conditions
\begin{equation}
\alpha_{k,0}=\hat v_k\quad (0\le k\le N).
                                    \label{6.15}
\end{equation}

Observe that \eqref{6.14} is a Riccati type system, i.e., its right-hand side is quadratic in the unknowns.
Standard facts of ODE's theory give us the following statement on the local existence of a solution:

\begin{lemma} \label{L6.4}
Given an integer
$N\in\N$
and
$\ep_c>0$,
there exists
$\delta_N=\delta_N(\ep_c)>0$
such that the following statement is true.

For every sequence
$\rho=(\rho_k)_{0\le k\le N}\in \C^{N+1}$
satisfying
$\sup_{0\le k\le N}|\rho_k|\le \ep_c^{-1}$,
system \eqref{6.14} has a unique solution
$$
{\tilde \alpha}_\rho=({\tilde \alpha}_{0,\tau,\rho},\dots,{\tilde \alpha}_{N,\tau,\rho})\in C^\infty((-\delta_N,\delta_N),\C^{N+1})
$$
satisfying the initial condition
${\tilde \alpha}_\rho(0)=\rho$.
\end{lemma}

The dependance of
$\delta_N$
on
$\ep_c$
is not designated explicitly since
$\ep_c$
is fixed in our further arguments.

We apply Lemma \ref{L6.4} to
$\rho=(\hat v_k)_{0\le k\le N}$
and then define
$$
\alpha_{\tau,v}(\theta)=\tilde \alpha_{0,\tau,\rho}+\sum_{k=1}^N (\tilde\alpha_{k,\tau,\rho}\,e^{ik\theta}
+ \overline{\tilde \alpha_{k,\tau,\rho}}\,e^{-ik\theta}).
$$
The path
$\alpha_{\tau,v}$
belongs to
$C^\infty((-\delta_N,\delta_N),C^\infty(\S))$
and
$\alpha_{0,v}=v$.
The path satisfies \eqref{6.13} for
$\tau\in (-\delta_N,\delta_N)$.
Due to Lemma \ref{L6.1}, the path also satisfies equation \eqref{6.12} for
$\tau\in (-\delta_N,\delta_N)$.

We are going to prove that
$\alpha_{\tau,v}\in \K_c$
for any
$\tau\in (0,\delta_N)$.
First we use Lemma \ref{L6.2} to obtain
$$
\int_{\S}\alpha_{\tau,v}\le \int_{\S}\alpha_{0,v}=\int_{\S}v\le c_0.
$$
Then we set
$T_v=\sup\{s\in (0,\delta_N)\mid \alpha_{s,v}\ \textrm{is a positive function}\}$.
By Lemma \ref{L6.3},
$$
\int_{\S}\alpha_{\tau,v}^{-1}=\int_{\S}\alpha_{0,v}^{-1}=\int_{\S}v^{-1}=2\pi
$$
for
$\tau\in[0,T_v)$.
Then we can apply Corollary \ref{C4.1} to obtain
$$
\zeta_{\alpha_{\tau,v}}(-1)\le \zeta_v(-1)\le c_1,\quad \zeta_{\alpha_{\tau,v}}(-2m)\le \zeta_v(-2m)\le c_{m+1}\ (m\in\N).
$$
Therefore
$\alpha_{\tau,v}\in \K_c$
for
$\tau\in(0,T_v)$.
In particular, by \eqref{6.2},
$$
\alpha_{\tau,v}\ge \ep_c\quad \textrm{for}\quad\tau\in(0,T_v).
$$
Hence we necessarily have
$T_v=\delta_N$.

Now, we are going to prove that the solution
$\alpha_{\tau,v}$
can be extended to all positive times
$\tau$.
To this end we introduce the one-parametric family of continuous maps
$$
\varphi^\tau:\K_c\cap {\mathcal F}_N\rightarrow \K_c\cap {\mathcal F}_N,\quad v\mapsto\alpha_{\tau,v}\quad(0\le\tau<\delta_N),
$$
where 
${\mathcal F}_N$ 
denotes the 
$(2N+1)$-dimensional subspace of 
$C^\infty(\S)$ 
consisting of smooth functions 
$f$ 
such that 
$\hat f_k=0$ 
for 
$|k|>N$.
By the well known group property of a solution to the Cauchy problem for ODE's,
$$
\varphi^{\tau_1+\tau_2}=\varphi^{\tau_1}\circ\varphi^{\tau_2}\quad\mbox{for}\quad \tau_1,\tau_2\in[0,\delta_N),\ \tau_1+\tau_2<\delta_N.
$$
Now, representing an arbitrary
$\tau\ge0$
as
$\tau=\tau_1+\dots+\tau_p$
with
$\tau_i\in[0,\delta_N)\ (1\le i\le p)$,
we define
$$
\varphi^\tau=\varphi^{\tau_1}\circ\dots\circ\varphi^{\tau_p}:\K_c\rightarrow \K_c.
$$
Then
$\alpha_{\tau,v}=\varphi^\tau(v)$
is well defined for all $\tau\ge0$.

Uniqueness of the solution
$\alpha$
to the Cauchy problem \eqref{6.11}--\eqref{6.13} follows from the local uniqueness of Lemma \ref{L6.4}.
\end{proof}

\subsection{Convergence as $\tau\to +\infty$}

\begin{theorem} \label{Th6.3}
Let a positive function
$a\in C^\infty(\S)$
satisfy the normalization condition \eqref{1.4}.
Let
$\alpha\in C^\infty([0,\infty),C^\infty(\S))$
be a deformation of
$a$
satisfying equation \eqref{1.8}. Let
$\K_c$
be a compact set in
$C^\infty(\S)$
defined by \eqref{6.1} for a sequence
$c=\{c_k\}_{k\in \N}$
of positive reals such that
$$
c_0\ge \hat a_0,\quad c_1\ge \zeta_a(-1),\quad c_{m+1}\ge \zeta_a(-2m)\ (m\in \N\b\{0\}).
$$
Then $\alpha$ possesses the following properties:

{\rm (1)}
for all
$k,m\in\N$
the estimate
\begin{equation}
\sup_{\tau\in [0,+\infty)}\left\|{\pa^k\alpha_\tau\over\pa \tau^k}\right\|_{C^m(\S)}\le C_{m,k}
                                      \label{6.16}
\end{equation}
holds with a constant
$C_{k,m}$
that depends on
$k,m$
and the constants
$c_0,\ldots, c_{k+m+2}$;

{\rm (2)}
$\alpha_\tau$
converges to the constant function
${\mathbf 1}$
in
$C^\infty(\S)$
as
$\tau\to +\infty$.
\end{theorem}

\begin{proof}
Repeating our arguments from the proof of Lemma \ref{L6.4}, we prove that
$\alpha_\tau\in \K_c$
for any
$\tau \in [0,+\infty)$.
Therefore \eqref{6.16} holds for
$k=0$.
A similar estimate holds for
$\H \alpha_\tau$
in place of
$\alpha_\tau$.
Then we prove the estimate \eqref{6.16} for any
$k$
by induction on
$k$
and by iterative differentiation of the equation \eqref{6.12}.

Now we prove the second property. Let
$\{\tau_k\}_{k\in\N}$
be an increasing sequence of positive reals such that
$\tau_k\to +\infty$
as
$k\to +\infty$.
Since
$\K_c$
is a compact in
$C^\infty(\S)$,
there exists a subsequence
$\{\tau_{k_n}\}_{n\in \N}$
such that
$\alpha_{\tau_{k_n}}$
converges in
$C^\infty(\S)$
to some function
$\alpha_\infty$.
We are going to prove that
$\alpha_{\infty}={\mathbf 1}$.
Since the limit
$\alpha_\infty$
is then unique, this would prove the second statement of the theorem.

Since
$\alpha_\tau>0$
and
$\int_{\S}\alpha_\tau^{-1}=2\pi$
for any
$\tau \ge 0$,
we have by the Bunyakovsky-Cauchy-Schwarz inequality
$$
1\le (2\pi)^{-2}\int_{\S}\alpha_\tau^{-1}\int_{\S}\alpha_\tau=(2\pi)^{-1}\int_{\S}\alpha_\tau.
$$
Therefore
\begin{equation}
\eta=\inf_{\tau\in[0,+\infty)}\int_{\S}\alpha_\tau\ge 2\pi.
                          \label{6.17}
\end{equation}
We also recall that
$\int_{\S}\alpha_\tau$
is a non-increasing function of
$\tau$
(see Lemma \ref{L6.2}). Hence
$$
\eta=\lim_{k\to +\infty}\int_{\S}\alpha_{\tau_{n_k}}=\int_{\S}\alpha_{\infty}.
$$

First assume that
$\alpha_\infty$
is not a constant function. Then
$$
\l \alpha_{\infty,+},\Lambda \alpha_{\infty,+} \r>0.
$$
Hence, by \eqref{6.10},
\begin{equation}
\left.{\pa \int_\S \alpha_\tau\over \pa \tau}\right|_{\tau=\tau_{n_k}}=-4\l \alpha_{\tau_{n_k},+},\Lambda \alpha_{\tau_{n_k},+} \r\to -4\l \alpha_{\infty,+},\Lambda \alpha_{\infty,+} \r<0\textrm{ as }\ k\to\infty.
                                       \label{6.18}
\end{equation}

Then using \eqref{6.16} we obtain that
\begin{equation}
\sup_{\tau\in (0,\infty)}\Big|{\pa^2 \int_\S \alpha_\tau\over \pa \tau^2}\Big|\le C
                                \label{6.19}
\end{equation}
with some positive constant
$C$
that depends on
$\K_c$.

Estimates \eqref{6.18} and \eqref{6.19} prove the existence of
$k_0\in \N$,
$\delta>0$
and
$r>0$
such that
$$
\int_\S \alpha_{\tau_{n_k}+\delta} \le \int_\S \alpha_{\tau_{n_k}}-r
$$
for any
$k\ge k_0$.
Hence
$\eta=-\infty$
since
$\int_\S \alpha_{\tau}$
is non-increasing in
$\tau$.
This contradicts \eqref{6.17}.

We have proved that
$\alpha_\infty$
is a constant function. The normalized condition \eqref{1.4} is preserved along the path
$ \alpha$,
and we obtain
$\alpha_\infty={\mathbf 1}$.
\end{proof}

\subsection{Final step}
We prove Theorem \ref{Th1.3}.
Let a positive function
$a\in C^\infty(\S)$
satisfy the normalization condition \eqref{1.4}.

Let us recall the algebraic definition \cite{MS} of zeta invariants
$\zeta_b(-2m)\ (m=1,1,\dots)$
for a positive function
$b\in C^\infty(\S)$:
$$
\zeta_b(-2m)=\sum\limits_{j_1+\dots +j_{2m}=0} N_{j_1\dots j_{2m}}\,{\hat b}_{j_1}{\hat b}_{j_2}\dots {\hat b}_{j_{2m}},
$$
where, for $j_1+\dots+ j_{2m}=0$,
\begin{equation}
\begin{aligned}
N_{j_1\dots j_{2m}}=\sum\limits_{n=-\infty}^\infty
\Big[&\left|n(n+j_1)(n+j_1+j_2)\dots(n+j_1+\dots+ j_{2m-1})\right|\\
&-n(n+j_1)(n+j_1+j_2)\dots(n+j_1+\dots +j_{2m-1})\Big].
\end{aligned}
                                               \label{6.20}
\end{equation}
There is only a finite number of nonzero summands on the right-hand side of \eqref{6.20} since the expression
$$
f(n)=n(n+j_1)(n+j_1+j_2)\dots(n+j_1+\dots+ j_{2m-1})
$$
is a polynomial of degree
$2m$
in
$n$ which takes positive values for sufficiently large
$|n|$.

Now the function
$a$
is positive. Hence there exists
$N_0\in\N$
such that for
$N\ge N_0$
$$
\sum_{|n|\le N}\hat a_k e^{i k\theta}>0\quad (\theta\in \R).
$$
We set
\begin{equation}
a^{(N)}(\theta)=c_N\sum_{|n|\le N}\hat a_ke^{ik\theta}\quad(N\ge N_0),
                                       \label{6.21}
\end{equation}
where
$c_N$
is determined by the normalized condition
\begin{equation}
\int_{\S}\Big(\sum_{|n|\le N}\hat a_ke^{ik\theta}\Big)^{-1}=2\pi c_N.
                                        \label{6.22}
\end{equation}
We also recall that
$\hat a_k=O(|k|^{-\infty})$
as
$k\to\infty$
since
$a$
is a smooth function.

By Kogan's formula \eqref{6.5}, \eqref{6.22} and the algebraic definition \eqref{6.20} of zeta invariants,
\begin{equation}
\lim_{N\to+\infty}\zeta_{a^{(N)}}(-1)=\zeta_a(-1),\quad \lim_{N\to+\infty}\zeta_{a^{(N)}}(-2m)= \zeta_a(-2m)\ (m=1,2,\dots).
                                              \label{6.23}
\end{equation}

Let
$N_1\ge N_0$
be large enough. By \eqref{6.23} we can choose
$\delta_{N_1} >0$
such that
\begin{eqnarray*}
\widehat{(a^{(N)})}_0&\le& c_0=\hat a_0+\delta_{N_1},\\
\zeta_{a^{(N)}}(-1)&\le&c_1=\zeta_a(-1)+\delta_{N_1},\\
\zeta_{a^{(N)}}(-2m)&\le&c_{m+1}=\zeta_a(-2m)+\delta_{N_1}
\end{eqnarray*}
for
$1\le m\le 2N_1$
and
$N\ge N_0$.

Now, consider the compact set
$\K_c^{(N)}$
in
$C^\infty(\S)$
defined by the sequence
$(c_m^{(N)})_{m\in \N}$
that is defined as follows:
\begin{equation}
c_m^{(N)}=c_m\ \mbox{for}\ 0\le m\le 2N_1\quad \textrm{and}\quad c_{k+1}^{(N)}=\zeta_{a^{(N)}}(-2k)\ \mbox{for}\ k\ge 2N_1+1.
                                             \label{6.24}
\end{equation}
We apply Theorem \ref{Th6.2} to
$a^{(N)}$
and
$\K_c^{(N)}$:
There exists a path
$\alpha^{(N)}\in C^\infty([0,\infty),C^\infty(\S))$
that converges in
$C^\infty(\S)$
to
${\mathbf 1}$
as
$\tau\to +\infty$
and satisfies \eqref{6.12}--\eqref{6.13} with the initial condition
$\alpha^{(N)}_0=a^{(N)}$.
Additionally, the estimate
\begin{equation}
\sup_{\tau\in [0,+\infty)}\Big\|{\pa^k\alpha_\tau^{(N)}\over\pa \tau^k}\Big\|_{C^{N_1-1}(\S)}\le C_{N_1,k}
                                                 \label{6.25}
\end{equation}
holds for any
$0\le k\le N_1-1$,
where the constant
$C_{N_1,k}$
depends on
$N_1$, $k$ and
constants
$c_0,\ldots, c_{2N_1+1}$
given in the definition of
$\K_c^{(N)}$.

Estimate \eqref{6.25} shows the existence of a subsequence
$(\alpha^{(N_k)})_{k\in  \N}$
that converges to some
$\alpha\in C^{N_1-2}([0,\infty),C^{N_1-2}(\S))$
in the space
$C^{N_1-2}([0,\infty),C^{N_1-2}(\S))$.
Passing to the limit in \eqref{1.8}, we see that
$\alpha$
solves \eqref{1.8} with the initial condition
$$
\lim_{N\to\infty}\alpha^{(N)}(0)=\lim_{N\to\infty}a^{(N)}=a
$$
(the limits are taken in $C^{N_1-2}(\S)$).
Since
$N_1$
is arbitrary  and since the
$\alpha^{(N)}$'s
do not actually depend on
$N_1$,
we obtain that
$$
\alpha\in C^\infty([0,\infty),C^\infty(\S)).
$$
The solution
$\alpha$
satisfies all statements of Theorem \ref{Th1.3}.\hfill $\Box$

\section{Concluding remarks}
In our previous work \cite{JS2}, the inequality
$\zeta_a(s)-2\zeta_R(s)\ge 0$
was proved for all real
$s$
satisfying
$|s|\ge1$.
The proof was based on inequalities \eqref{5.16}--\eqref{5.17} and essentially used the convexity of the function
$x\mapsto x^s\ (x\ge0)$
for
$s\ge1$. Together with \eqref{5.16}--\eqref{5.17}, the convexity gives
$$
\l (\Lambda_a+P_0)^s \phi_n,\phi_n\r\ge |n|^s,\quad \l(\Lambda_a+P_0)^{-s}\phi_n,\phi_n\r\ge |n|^{-s}\quad (n\in \Z\b\{0\},\ s\ge 1).
$$
These inequalities are definitely wrong for
$s\in (0,1)$.
Otherwise we would have
$$
\l \ln(\Lambda_a+P_0) \phi_n,\phi_n\r= \ln |n|\quad (n\in  \Z\b\{0\}).
$$
But a computations in a neighborhood of
$a={\mathbf 1}$
shows that the inequalities do not hold in the general case.

In the current work, we have developed an alternative approach for proving the inequality
$\zeta_a(s)-2\zeta_R(s)\ge0$
for all
$s\in\R$.

\bigskip

Let us reproduce equation \eqref{1.8}
\begin{equation}
{\pa\alpha_\tau\over \pa\tau}=-\alpha_\tau (\Lambda \alpha_\tau)+(\H\alpha_\tau) (D\alpha_\tau)
                                                 \label{7.1}
\end{equation}
together with the initial condition
\begin{equation}
\alpha_0=a.
                                                 \label{7.2}
\end{equation}
Observe that \eqref{7.1} is a Riccati type equation with non-local quadratic terms. We have proved the global existence of a solution to the Cauchy problem \eqref{7.1}--\eqref{7.2} at least for a positive function
$a\in C^\infty(\S)$.
But the corresponding uniqueness question remains open.

Another Riccati type equation with non-local quadratic terms is well known in the {\it layer stripping method} for Electrical Impedance Tomography, see \cite{CNS} and references therein. To our knowledge, the uniqueness and global existence of a solution to the Cauchy problem for the latter equation are proved in the radially symmetric case only. Nevertheless, in more general cases, some numerical methods are developed  which are based on the equation.

\end{document}